\documentclass[leqno,a4paper,12pt]{amsart}
\pdfoutput=1
\usepackage{amssymb}
\usepackage{enumerate,mathrsfs}
\usepackage{graphicx,tikz,tikz-cd}
\usetikzlibrary{positioning}

\theoremstyle{plain}
\newtheorem{theorem}{\indent\sc Theorem}[section]
\newtheorem{lemma}[theorem]{\indent\sc Lemma}
\newtheorem{corollary}[theorem]{\indent\sc Corollary}

\theoremstyle{definition}
\newtheorem{definition}[theorem]{\indent\sc Definition}
\newtheorem{remark}[theorem]{\indent\sc Remark}
\newtheorem{example}[theorem]{\indent\sc Example}

\newtheorem{conjecture}[theorem]{\indent\sc Conjecture}

\newcommand{\FF}{\mathbb{F}}
\newcommand{\NN}{\mathbb{N}}
\newcommand{\ZZ}{\mathbb{Z}}

\newcommand{\RR}{\mathbb{R}}
\newcommand{\CC}{\mathbb{C}}
\newcommand{\HH}{\mathbb{H}^3}
\newcommand{\Len}{\mathcal{L}}
\newcommand{\KK}{\mathscr{K}\mkern-2mu}
\newcommand{\KKsub}[1]{\mathscr{K}_{\mkern-2mu{#1}}}
\newcommand{\triang}{\mathcal{T}}
\newcommand{\Id}{\mathrm{Id}}
\newcommand{\mapswith}[1]{\xrightarrow{#1}}
\newcommand{\slE}{\mathfrak{e}}
\newcommand{\slF}{\mathfrak{f}}
\newcommand{\slH}{\mathfrak{h}}
\newcommand{\basis}[1]{\underline{#1}}
\newcommand{\nword}[2]{$#1$\nobreakdash--#2}
\newcommand{\PSL}{PSL_2\CC}
\newcommand{\slC}{\mathfrak{sl}_2\CC}
\newcommand{\glvar}{\mathcal{V}_{\mathcal{T}}^+}
\newcommand{\GIT}{/\mkern-6mu/}
\newcommand{\dby}[1]{\frac{\partial}{\partial #1}}
\newcommand{\paraZ}{\Sigma_\triang}
\newcommand{\paraX}{\Sigma'_{\triang,\theta}}
\newcommand{\Choibasis}{\basis{d^\bullet}}
\newcommand{\auxiliaryMap}{\psi}
\newcommand{\mathstos}[2]{\tikz[baseline=({u.south})]{
	\node[inner sep=0.3ex] (u) at (0,0) {$#1$};
	\node[inner sep=0.3ex,below=0.2ex of {u.south west}, anchor=north west] {$#2$};}%
}
\DeclareMathOperator{\Ad}{Ad}
\DeclareMathOperator{\Imag}{Im}
\DeclareMathOperator{\Real}{Re}
\DeclareMathOperator{\Kernel}{Ker}
\DeclareMathOperator{\Coker}{Coker}

\DeclareMathOperator{\linspan}{span}
\DeclareMathOperator{\Isom}{Isom}
\DeclareMathOperator{\diag}{diag}
\DeclareMathOperator{\Tor}{\mathbb{T}}
\DeclareMathOperator{\ATor}{\mathbb{T}_{\Ad}}

\DeclareMathOperator{\vol}{vol}
\DeclareMathOperator{\Hom}{Hom}
\DeclareMathOperator{\tr}{tr}
\DeclareMathOperator{\Or}{Or}
\DeclareMathOperator{\Aff}{Aff}
\DeclareMathOperator{\interior}{int}
\DeclareMathOperator{\const}{const}
\begin{document}
\title[Infinitesimal gluing equations and adjoint torsion]{Infinitesimal gluing equations and the adjoint hyperbolic Reidemeister torsion}
\author[R. Siejakowski]{Rafa\l{} Siejakowski$^*$}
\thanks{$^{*}$The author was supported by the Singapore Ministry of
Education tier 1 grant no.~SPMS-RG66-10 and by
grant no. 2018/12483-0 of the S\~ao Paulo Research Foundation (FAPESP)}
\address{Instituto de Matem\'atica e Estat\'istica \endgraf
Universidade de S\~ao Paulo \endgraf
Rua do Mat\~ao 1010, 05508-090 S\~ao Paulo, SP \endgraf
Brazil}
\email{rafal@ime.usp.br}
\begin{abstract}
We establish a link between the derivatives of Thurston's hyperbolic gluing equations on
an ideally triangulated finite volume hyperbolic 3-manifold and the cohomology of the
sheaf of infinitesimal isometries.
This provides a geometric reformulation of the non-abelian Reidemeister torsion
corresponding to the adjoint of the monodromy representation of the hyperbolic structure.
These results are then applied to the study of the `1\nobreakdash-loop Conjecture' of
Dimofte--Garoufalidis, which we generalize to arbitrary 1\nobreakdash-cusped hyperbolic
3\nobreakdash-manifolds.
We verify the generalized conjecture in the case of the sister manifold of the
figure-eight knot complement.
\end{abstract}
\maketitle

\section{Introduction}
This paper aims to establish a link between an infinitesimal version of Thurston's
hyperbolic gluing equations and the adjoint Reidemeister torsion of a finite volume
hyperbolic \nword{3}{manifold}.
This goal is realized in two steps.
Firstly, we explore the cohomological meaning of the derivatives
of the edge consistency and completeness equations on a
positively oriented geometric ideal triangulation~$\triang$.
The basic idea is to interpret the complex tangent space $T_z\CC_{\Imag>0}$ as the space of
infinitesimal deformations of the geometry of a hyperbolic ideal tetrahedron in~$\HH$ with the shape
parameter $z\in\CC_{\Imag>0}$.
When several ideal tetrahedra are glued together to form a geometric triangulation~$\triang$,
the deformations of the individual tetrahedra induce infinitesimal deformations of the geometry of the
resulting hyperbolic \nword{3}{manifold} $M$.
We describe these deformations as first cohomology classes
with coefficients in the sheaf of infinitesimal isometries of $M$ or a restriction thereof.

Secondly, since the bundle of infinitesimal isometries on a hyperbolic \nword{3}{manifold}~$M$
is isomorphic to the flat, rank~$3$ complex
vector bundle defined by the adjoint of the monodromy representation~$\pi_1(M)\to\PSL$ of the
hyperbolic structure, we can restate Porti's construction~\cite{porti1997} of the combinatorial
\emph{adjoint hyperbolic torsion}~$\ATor(M): H_1(\partial_\infty M;\ZZ)\to\CC^*/\{\pm1\}$
in terms of the sheaf of germs of Killing vector fields on $M$.
This provides a geometric interpretation of the adjoint torsion.

Finally, we show how these insights can be used to calculate the torsion invariant~$\ATor(M)$
in terms of the shapes for an ideal triangulation of $M$.
In particular, our method correctly reproduces the main factor of the
`\nword{1}{loop} invariant', a conjectural expression for the adjoint
torsion given by Dimofte and Garoufalidis in \cite{tudor-stavros}.

\subsection{Infinitesimal gluing equations}
The hyperbolic gluing equations were first introduced by W. Thurston~\cite{thurston-notes}.
Neumann and Zagier~\cite{neumann-zagier} discovered a symplectic property of these equations
which was further studied by Neumann in \cite{neumann1990} and more recently reinterpreted
by Dimofte and van~der~Veen~\cite{tudor-roland-spectral}
in terms of intersection theory on certain branched double covers.
Within mathematical physics, gluing equations have been used to
construct quantum Chern--Simons theories on
ideal triangulations, with the symplectic structure serving as the starting
point for geometric quantization; see in particular Dimofte~\cite{dimofte-quantum-Riemann}
and Dimofte--Garoufalidis~\cite{tudor-stavros}.

Suppose that $\triang$ is an abstract ideal triangulation of a connected orientable
open \nword{3}{manifold}~$M$ with $k$ ideal vertices, the links of which are all tori.
Denote by $N$ be the number of tetrahedra, and hence also of edges, of $\triang$.
Choi~\cite{youngchoi} reformulated the hyperbolic edge consistency equations
in terms of a single map $g:\CC_{\Imag>0}^N\to(\CC^*)^N$, the domain of which is thought
of as the space of shape parameters of $N$ positively oriented ideal tetrahedra.
By definition, $g$ assigns to any \nword{N}{tuple} of shapes
the \nword{N}{tuple} of their products about the edges of $\triang$,
so that the consistency equations read $g(z)=\mathbf{1}$.
Consider also a collection $\theta=(\theta_1,\dotsc,\theta_k)$ of oriented, homotopically
nontrivial curves in normal position with respect to $\triang$, one in each vertex link.
The log-parameters along the constituent curves of $\theta$, defined by the way of~\cite{neumann-zagier},
give rise to a map~$u=u_\theta:\CC_{\Imag>0}^N\to\CC^k$
so that the completeness condition becomes $u(z)=0$.
With these notations, the results of Neumann--Zagier~\cite{neumann-zagier} and Choi~\cite{youngchoi}
imply the existence of the tangential exact sequence
of `infinitesimal gluing equations'
\begin{equation}\label{intro:Choi-seq}
	0 \to T_{u(z)} U \mapswith{Dy} T_z \CC_{\Imag>0}^N
	\mapswith{Dg} T_1 (\CC^*)^N \mapswith{Dp} \CC^k \to 0,
\end{equation}
where $y=y_\theta: U\to g^{-1}(\mathbf{1})$ is a local analytic inverse of $u$ and $p$
is a monomial map defined by the incidences of the edges of $\triang$ to the ideal vertices;
see Section~\ref{gluing-section} below for the details.

\subsection{Infinitesimal hyperbolic isometries}
Recall that a \emph{Killing field} on a Riemannian manifold $M$
is a vector field whose flows are local isometries of $M$.
We denote the Lie algebra of all Killing fields on $M$ by $\KK(M)$.
The assignment of the space~$\KK(U)$ to any
open set~$U\subset M$ defines a sheaf~$\KKsub{M}$ on $M$, called
the \emph{sheaf of (germs of) Killing vector fields}.
Geometrically, the sheaf~$\KKsub{M}$ can be viewed as the sheaf of
local \emph{infinitesimal isometries} of $M$.
The cohomology of $\KKsub{M}$ is therefore closely related to the deformation theory of geometric
structures, as explained in the case of hyperbolic \nword{3}{manifolds} by Hodgson and Kerckhoff
in \cite{craig-steve}.
We refer to \cite{nomizu,matsushima-murakami} for more information on Killing vector fields
in the general setting of symmetric Riemannian manifolds.

Suppose that the triangulation $\triang$ of $M$ is geometric, so that $M$ is endowed with
a finite-volume hyperbolic structure, not necessarily complete.
This in particular fixes the sheaf~$\KKsub{M}$.
Consider the closed subspace $M_0\subset M$ resulting from the removal of disjoint
open neighborhoods of all edges of $\triang$ from $M$
and let $\KKsub{M_0}$ be the restriction of $\KKsub{M}$ to $M_0$.
In particular, we obtain the short exact sequence
\begin{equation}\label{intro:sheaves}
	0\to \KKsub{M,M_0}\to\KKsub{M}\to\KKsub{M_0}\to 0
\end{equation}
of sheaves on $M$, cf.~\cite[Th.~2.9.3]{godement}.
The kernel sheaf~$\KKsub{M,M_0}$ can be understood as follows:
for any open set~$V\subset M$,  $\KKsub{M,M_0}(V)=0$ whenever $V\cap M_0\neq\varnothing$;
otherwise, $\KKsub{M,M_0}(V)=\KKsub{M}(V)$.
Similarly, the sheaf $\KKsub{M_0}$ can be defined over $M$
by $\KKsub{M_0}(V)=\KKsub{M}(V\cap M_0)$ for all open subsets~$V\subset M$.
We refer to \cite[\S2.9]{godement} for more details.

The short exact sequence~\eqref{intro:sheaves} leads to a
cohomology long exact sequence, which in this case reduces to the four non-zero terms
\begin{equation}\label{intro:cohomology-four}
	0 \to H^1(M;\KKsub{M}) \to H^1(M;\KKsub{M_0}) \to H^2(M;\KKsub{M,M_0}) \to H^2(M;\KKsub{M}) \to 0.
\end{equation}
The cohomological meaning of the gluing equations is then established by the following theorem.
\begin{theorem}\label{intro:embedding}
	At a generic hyperbolic structure on $M$, with all ends incomplete, the acyclic
	complex $\eqref{intro:Choi-seq}$ embeds as a subcomplex of $\eqref{intro:cohomology-four}$.
\end{theorem}\noindent
The above theorem is stated in more detail as Theorem~\ref{commutative-diagram-theorem} below;
cf. also Theorem~4.3.1 in \cite{rs-phd}.

Recall that the leftmost map~$Dy$ of the exact sequence~\eqref{intro:Choi-seq} depends
on the chosen multicurve~$\theta$.
This is because $y$ was defined as a local analytic inverse of the map~$u=u_\theta$,
the log-parameter along $\theta$.
At the level of cohomology, this dependence is expressed by the following theorem.
\begin{theorem}\label{intro:cup-theorem}
	The unique map $c$ which makes the diagram
	\begin{equation}\label{intro:cup-diagram}
		\begin{tikzcd}
			& T_uU \arrow{r}\arrow{d}{\cong}
			& H^1(M;\KKsub{M})\arrow{r}
			& H^1(M;\KKsub{\partial_\infty M})\arrow{d}{\textstyle c}\\
			& \CC^k \arrow{r}
			& H^2(M;\KKsub{M})\arrow{r}
			& H^2(M;\KKsub{\partial_\infty M})
		\end{tikzcd}
	\end{equation}
	commutative is given by $c(x)=x\smile[\theta]^*$,
	where $[\theta]^*\in H^1(\partial_\infty M;\ZZ)$ is the Poincar\'e dual of the homology
	class of $\theta$.
\end{theorem}\noindent
In the diagram~\eqref{intro:cup-diagram}, the horizontal maps on the left side are the embeddings
(in fact, isomorphisms) given by Theorem~\ref{intro:embedding} and $\partial_\infty M$ denotes the
toroidal boundary at infinity of $M$.
The vertical isomorphism on the left is the trivial one, induced by the embedding~$U\subset\CC^k$.
We refer to Theorem~\ref{log-parameters-cups} below for a precise statement and to \cite[Section~4.4]{rs-phd}
for an extended discussion.

\subsection{Geometric approach to the adjoint hyperbolic torsion}
The action of $\PSL$ on $\HH$ by orientation-preserving
hyperbolic isometries identifies the space~$\KK(\HH)$ of global Killing fields
with the Lie algebra~$\slC$.
Similarly, the sheaf~$\KK=\KKsub{M}$ on an orientable hyperbolic \nword3manifold~$M$
is locally modeled after $\slC$, which we consider here with the discrete topology.
To understand this relationship algebraically,
assume that $M$ is connected and consider a monodromy representation~$\varrho:\pi_1(M)\to\PSL$
of the hyperbolic structure on $M$.
Let $E=E_{\Ad\varrho}$ be the rank~$3$ vector bundle on $M$
defined by
\begin{equation}\label{explicit-E}
	E = \widetilde{M}\times_{\pi_1(M)}\slC,
\end{equation}
where $\pi_1(M)$ acts on the universal covering space~$\widetilde{M}$ by deck transformations and on
$\slC$ via~$\Ad\varrho$.
By a theorem of Matsushima--Murakami \cite[Theorem~8.1]{matsushima-murakami},
the sheaf $\KKsub{M}$ is isomorphic to the sheaf~$\Gamma(E)$ of continuous sections of $E$.
This isomorphism naturally endows $\KKsub{M}$ with the structure of a locally constant
sheaf of \emph{complex} vector spaces.
In particular, the monodromy of $E$ can be understood as analytic continuation of locally
defined Killing fields, as studied by Nomizu~\cite{nomizu}.

The torsion invariant~$\ATor$ was constructed by Porti~\cite{porti1997}
as a combinatorial twisted Reidemeister torsion of $M$, where the twisting
comes from the action of $\pi_1(M)$ on $\slC$ via $\Ad\varrho$.
In general, adjoint torsion invariants can be interpreted as top degree differential forms
on the regular locus of character varietes \cite{porti1997,dubois-SU2,frohman-kania}
of special linear or projective groups.
Using the isomorphism $\KKsub{M}\cong\Gamma(E)$ we are able to express~$\ATor(M)$
in terms of cellular, simplicial or \v{C}ech cochains with coefficients in $\KKsub{M}$.
In particular, the normalization of torsion introduced by Porti can be recovered from our
geometric interpretation and from Theorem~\ref{intro:cup-theorem}.

\subsection{Application to the 1-loop Conjecture}
Using the methods of mathematical physics, Dimofte and Garoufalidis~\cite{tudor-stavros}
constructed a formal power series~$\mathcal{Z}_\triang(\hbar)$ associated to a geometric ideal
triangulation~$\triang$ of a hyperbolic knot complement~$M=S^3\setminus K$.
The coefficients in the power series~$\mathcal{Z}_\triang(\hbar)$ are complex numbers
defined as weighted sums over Feynman diagrams with an increasing number of loops,
and consequently called `\nword{n}{loop}' coefficients.
They are determined by the combinatorics of~$\triang$ and by the shape
parameter solutions of the gluing equations;
see \cite{garoufalidis-sabo} for an extended discussion.
It is not known whether
the \nword{n}{loop} coefficients are topological invariants of $M$ for $n>0$.

A focal point of \cite{tudor-stavros} is the conjectural `\nword{1}{loop}
invariant'~$\tau_\triang\in\CC/\{\pm1\}$ determined by the \nword{1}{loop}
coefficient of $\mathcal{Z}_\triang(\hbar)$.
If the power series~$\mathcal{Z}_\triang(\hbar)$ is an asymptotic
expansion of the Kashaev invariant~\cite{kashaev-volume-conjecture} at least to the first order,
then the Generalized Volume Conjecture~\cite{gukov-murakami}
would predict that the \nword{1}{loop} invariant coincides
with the torsion~$\ATor(M,\mu)$ corresponding to the knot-theoretic meridian $\mu$.
The equality $\tau_\triang=\ATor(M,\mu)$
is the content of the \emph{1-loop Conjecture} \cite[Conjecture~1.8]{tudor-stavros},
which we state as Conjecture~\ref{1-loop} below.

If true, the \nword{1}{loop} Conjecture would provide a particularly simple and explicit formula
for $\ATor$ in terms of a geometric ideal triangulation.
We show that the main term of this formula arises naturally from
the factorization of torsion induced by \eqref{intro:sheaves}.
This computation, presented in Section~\ref{torsion-section} below,
implies in particular the non-vanishing of $\tau_\triang$.

Compared to the original statement in \cite{tudor-stavros}, our version of the \nword{1}{loop}
Conjecture is adapted to work with any system of homotopically
non-trivial simple closed curves, not necessarily knot meridians.
Hence, we can generalize the conjecture to all triangulated, orientable one-cusped hyperbolic
\nword{3}{manifolds}.
Finally, we verify the generalized conjecture for the minimal
triangulation of the figure-eight sister manifold (\texttt{m003} in the
SnapPea cusped census~\cite{snappy}).
This manifold is not a complement of a knot in an integral homology sphere.

\par\noindent\textbf{Acknowledgements.}
This paper presents the main results of the PhD thesis~\cite{rs-phd} of the author, written
under the direction of Andrew Kricker.
The author would like to express his gratitude to Andrew Kricker for his support and guidance.
The author also wishes to thank John Hubbard, Craig Hodgson, Joan Porti and Tudor Dimofte for
their interest in his work and for many helpful conversations.

\section{Hyperbolic gluing equations}
\subsection{Ideal triangulations and Thurston's equations}\label{gluing-section}
Let $M$ be an orientable connected \nword{3}{manifold} homeomorphic to the interior of a
compact manifold $\overline{M}$ whose boundary is a union of $k>0$ tori.
Suppose that $\triang$ is an ideal triangulation of $M$ with $N$ tetrahedra.
By Euler characteristic considerations, the number of edges of $\triang$ is also $N$.
We fix an arbitrary numbering of the tetrahedra and of the edges of $\triang$ by integers
$\{1,\dotsc,N\}$.
We also label the toroidal ends with integers~$\{1,\dotsc,k\}$.
The hyperbolic gluing equations~\cite{thurston-notes,neumann-zagier} on $\triang$ can
then be written as
\begin{equation}\label{Thurstons-consistency}
	\prod_{j=1}^N {z_j}^{G_{ij}} {z'_j}^{G'_{ij}} {z''_j}^{G''_{ij}} = 1
	\ \text{for}\ i\in\{1,\dotsc,N\} = \{\text{edge indices}\},
\end{equation}
\par\noindent
where $z_j$, $z'_j=1/(1-z_j)$ and $z''_j = 1-1/z_j$ are the three shape parameters associated
to the $j$-th tetrahedron (see Figure~\ref{fig:Tetrahedron}).
We assemble the incidence numbers occurring as exponents in
\eqref{Thurstons-consistency} into the integer matrices~$G=[G_{ij}]$,
$G'=[G'_{ij}]$, $G''=[G''_{ij}]$.

As discussed in \cite{neumann-zagier}, the equations expressing the completeness of a hyperbolic
structure can be written in a similar form.
Suppose that $\theta=\{\theta_l\}_{l=1}^k$ is a collection of homotopically non-trivial oriented
simple closed peripheral curves, one at each end of $M$.
Then the logarithmic form of the completeness equations is
\begin{equation}\label{Thurstons-completeness}
	\sum_{j=1}^N C_{lj} \log z_j + C'_{lj} \log z'_j + C''_{lj} \log z''_j = 0,
	\ \text{for}\ l\in\{1,\dotsc,k\} = \{\text{indices of ends}\}.
\end{equation}
\par\noindent
In the above equation, ``$\log$'' denotes the standard branch of the logarithm on the upper
halfplane $\CC_{\Imag>0}=\{z\in\CC: \Imag z>0\}$
and the coefficients depend on the chosen representative of the free homotopy class of $\theta$
in normal position with respect to $\triang$ by the way of \cite{neumann-zagier}.
As before, we assemble these coefficients into $k\times N$ integer matrices
$C=[C_{lj}]$, $C'=[C'_{lj}]$ and $C''=[C''_{lj}]$.

It is well known~\cite{thurston-notes,youngchoi} that a solution of the
equations~\eqref{Thurstons-consistency}
in $\CC_{\Imag>0}$ turns $\triang$ into a positively oriented geometric triangulation and endows
$M$ with a hyperbolic structure.
If, in addition, the completeness condition \eqref{Thurstons-completeness} holds, the resulting
hyperbolic structure is complete.
We remark that Neumann and Zagier~\cite{neumann-zagier} impose completeness conditions
on a collection of oriented curves forming a \nword{\ZZ}{basis} of $H_1(\partial\overline{M},\ZZ)$.
However, if the triangulation is positively oriented, then
a result of Choi~\cite[Corollary~4.14]{youngchoi} implies that it suffices to impose
the completeness condition on only one nontrivial curve per end.

\subsection{The tangential gluing complex}
For the remainder of the section, we assume that $\triang$ admits a positively
oriented solution $z_*\in\CC_{\Imag>0}^N$ which recovers the unique complete hyperbolic structure
on $M$.
We now summarize certain results concerning the derivatives of the gluing equations,
due to Neumann--Zagier~\cite{neumann-zagier} and Choi~\cite{youngchoi}.

Define the map $g=g_\triang$ by the left-hand sides
of \eqref{Thurstons-consistency}:
\begin{equation}\label{def-g}
	g:\CC_{\Imag>0}^N \to (\CC^*)^N,
	\qquad g(z_1,\dotsc,z_N) =
	\Bigl(\prod_{j=1}^N {z_j}^{G_{ij}} {z'_j}^{G'_{ij}} {z''_j}^{G''_{ij}}\Bigr)_{i=1}^N.
\end{equation}
The set $\glvar := g^{-1}(\mathbf{1})$ is then called the
\emph{positive gluing variety} of the triangulation $\triang$.
For every $i\in\{1,\dotsc,N\}$ and $l\in\{1,\dotsc,k\}$,
denote by $e_i$ the $i$-th edge and by $v_l$ the $l$-th ideal vertex of $\triang$.
Let $K_{li}\in\{0,1,2\}$ be the number of ends of $e_i$ incident to $v_l$,
without any regard for orientations.
Following Choi~\cite{youngchoi}, we may define the monomial map $p$ by
\begin{equation}\label{def-p}
p : (\CC^*)^N\to(\CC^*)^k,
\qquad p(x_1,\dotsc,x_N) = \Bigl(\prod_{i=1}^N x_i^{K_{li}}\Bigr)_{l=1}^k.
\end{equation}
Choi then proves~\cite[Theorem~3.4]{youngchoi} that for any $z\in\glvar$,
there is an exact sequence
\begin{equation}\label{Choi-first}
	0 \to T_z\glvar \to T_z \CC_{\Imag>0}^N \mapswith{Dg} T_1 (\CC^*)^N \mapswith{Dp} \CC^k \to 0
\end{equation}
given by the holomorphic derivatives of the maps defined above,
where the last non-zero term $T_1 (\CC^*)^k$ has been trivially identified with $\CC^k$.
Denote by $u_l=u_l(z)$ the log-parameter along the peripheral curve
$\theta_l$ for $l\in\{1,\dotsc,k\}$;
explicitly, we have
\[
	u_l(z_1,\dotsc,z_N) = \sum_{j=1}^N C_{lj} \log z_j + C'_{lj} \log z'_j + C''_{lj} \log z''_j.
\]
Neumann--Zagier~\cite[\S4]{neumann-zagier} proved that there exists a neighborhood
of $z$ in $\glvar$ on which the log-parameters $\{u_l\}_{l=1}^k$ form a holomorphic
coordinate chart.
Denote by $y=y_\theta$ the inverse of this chart:
\begin{equation}\label{definition-of-y}
	y: U\to\glvar,\qquad
	y(u_1,\dotsc,u_k) = \bigl(z_1(u_1,\dotsc,u_k),\dotsc,z_N(u_1,\dotsc,u_k)\bigr),
\end{equation}
where $U\subset\CC^k$ is a sufficiently small neighborhood of the origin;
note that $y(0)=z_*$ by definition.
With these notations, we may replace \eqref{Choi-first} with the exact sequence
\begin{equation}\label{Choi-choice}
	0 \to T_{u(z)} U \mapswith{Dy} T_{z} \CC_{\Imag>0}^N
	\mapswith{Dg} T_1 (\CC^*)^N \mapswith{Dp} \CC^k \to 0.
\end{equation}

\subsection{Coordinates on character varietes via gluing equations}
\begin{figure}[tb]\centering
	\begin{tikzpicture}
		\node[anchor=south west,inner sep=0] (image) at (0,0)
		    {\includegraphics[height=55mm]{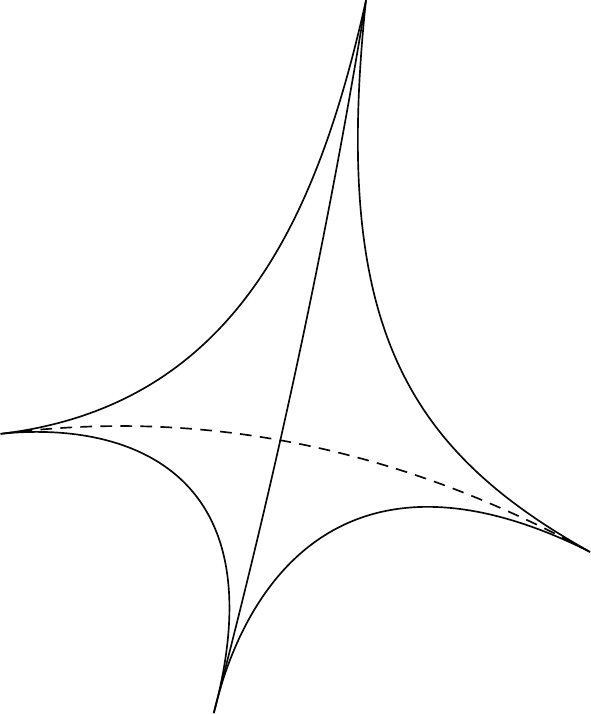}};
		\begin{scope}[x={(image.south east)},y={(image.north west)},
				every node/.style={inner sep=0, outer sep=0}]
			\node[anchor=south] at (0.3, 0.51) {$z$};
			\node[anchor=east] at (0.505, 0.5) {$z'$};
			\node[anchor=south west] at (0.705, 0.46) {$z''$};
			\node[anchor=north] at (0.65, 0.27) {$z$};
			\node[anchor=south] at (0.63, 0.37) {$z'$};
			\node[anchor=north east] at (0.32, 0.31) {$z''$};
		\end{scope}
	\end{tikzpicture}\hspace{0.1\textwidth}
	\includegraphics[height=55mm]{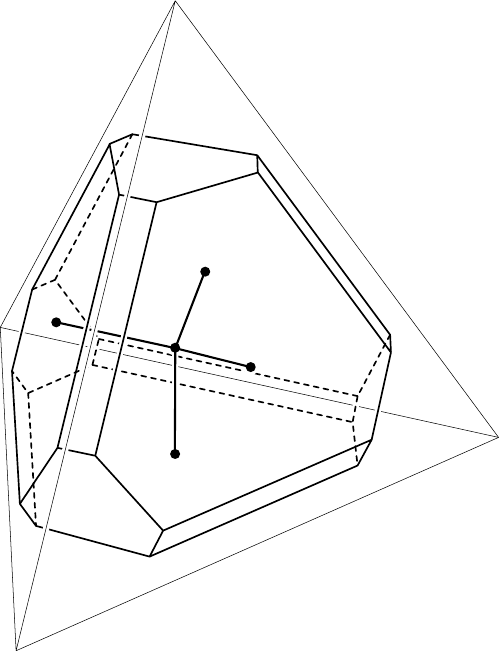}
	\caption{
	\textsc{Left:} The labeling convention for the shape parameters.
	\textsc{Right:}
	A \emph{doubly truncated tetrahedron} results from shaving off neighborhoods
	of the edges of a tetrahedron in addition to truncating its vertices.
	$M_0\subset M$ is homotopy equivalent to the topological space obtained from the ideal
	triangulation 	$\triang$ by replacing its tetrahedra with doubly truncated tetrahedra.
	Hence, $M_0$ retracts onto the graph given by the union of `tetrapod' graphs shown in the figure.
	}\label{fig:Tetrahedron}
\end{figure}
We now wish to summarize the tangential properties of the parametrizations of character
varietes induced by the hyperbolic shapes.
While our discussion focuses on the case of geometric ideal triangulations,
most of the results stated here hold more generally for algebraic solutions
$z\in(\CC\setminus\{0,1\})^N$ of Thurston's gluing equations~\eqref{Thurstons-consistency}.

We fix the orientation of $M$ once and for all and assume that the geometric
triangulation~$\triang$ is positively oriented.
The tori forming the boundary $\partial\overline{M}$ are oriented using the convention
`outward facing normal vector in the last position'.
For every edge $e_i$ of $\triang$, $1\leq i\leq N$, we may choose an open
tubular neighborhood $E_i\supset e_i$
in such a way that $E_i\cap E_j=\varnothing$ for $i\neq j$.

\begin{definition}\label{def-M0}
We define $M_0 = M_0(\triang) = M\setminus\bigcup_{i=1}^N E_i$.
\end{definition}
We remark that the space $M_0$ is called a `manifold with defects' in \cite{tudor-roland-spectral}.
In general,
$M_0\subset M$ is a handlebody which deformation-retracts onto the union of `tetrapod' graphs
inscribed into the tetrahedra of $\triang$, as depicted on the right panel of
Figure~\ref{fig:Tetrahedron}.
A geometric version of this construction was used in
\cite{constantino-frigerio} to study ideal triangulations.

Since the handlebody $M_0$ does not contain any edges of $\triang$, every collection of positively
oriented shape parameters $z=(z_1,\dotsc,z_N)\in\CC_{\Imag>0}^N$ determines a hyperbolic
structure on $M_0$; this structure extends to $M$ if and only if $z\in\glvar$.
Let $\varrho_z\in\Hom(\pi_1(M_0), \PSL)$ be a monodromy representation of the hyperbolic
structure induced on $M_0$ by the shape parameters $z$.
While $\varrho_z$ itself is only defined up to conjugation, it makes sense to talk
about the image of~$\varrho_z$ in the $\PSL$ \emph{character variety}~$X(\pi_1(M_0),\PSL)$.
We refer the reader to \cite{acuna-montesinos,heusener-porti} for more information on
character varieties.

\begin{definition}\label{def-paraZX}
We define
	\begin{equation}\label{shape-parametrization}
		\paraZ: \CC_{\Imag>0}^N\to X(\pi_1(M_0),\PSL),\qquad\paraZ(z) = [\varrho_z],
	\end{equation}
	where $X(\pi_1(M_0),\PSL)=\Hom(\pi_1(M_0), \PSL)\GIT\PSL$ is the \nword{\PSL}{character} variety
	of $\pi_1(M_0)$.
	We also define
	\[
		\paraX: U\to X(\pi_1(M),\PSL),
		\qquad \paraX(u) = \paraZ \circ y_\theta(u).
	\]
\end{definition}

\begin{lemma}
	Let $z_*\in\CC_{\Imag>0}^N$ be a positively oriented solution of edge consistency and
	completeness equations on $\triang$.
	Then $z_*$ has a neighborhood~$V\subset\CC_{\Imag>0}^N$ such that
	$\paraZ(V)$ consists
	only of regular points. Moreover, $\paraZ|_V$ is analytic.
\end{lemma}
\begin{proof}
As $M_0$ has the homotopy type of a graph, we see that $\pi_1(M_0)\cong F_{N+1}$,
the free group of rank $N+1$.
We assumed that $M$ is orientable, so $N\geq 2$ and thus the rank of $\pi_1(M_0)$
is at least three.
By a result of Heusener--Porti~\cite[Proposition~5.8]{heusener-porti},
a \nword{\PSL}{representation} $\varrho_z$ maps to a regular point of $X(\pi_1(M_0),\PSL)$
if and only if the adjoint representation $\Ad\varrho_z$ is irreducible.
Since $\Ad\varrho_{z_*}$ is irreducible and the GIT quotient map
$\Hom(\pi_1(M_0),\PSL)\mapswith{/\mkern-4mu/}X(\pi_1(M_0),\PSL)$ is an analytic submersion
at regular points of the character variety, the result follows.
\end{proof}

In particular, the above lemma implies that $\paraX$ is an analytic map,
provided that the Dehn surgery parameter space~$U\subset\CC^k$ is taken to be small enough.
This parametrization was first studied by Neumann--Zagier~\cite[\S4]{neumann-zagier}.

\begin{lemma}\label{injectivity-lemma}
	If $u=(u_1,\dotsc,u_k)\in U$ satisfies $0<|u_l|<\pi$ for all $l$, then
	the derivative
	\[
		D\paraX: T_u U\to T_{[\varrho]}X(\pi_1(M),\PSL),
		\quad \text{where}\ [\varrho]=[\varrho_{y(u)}],
	\]
	is an isomorphism.
\end{lemma}
\begin{proof}
It is well known~\cite{neumann-zagier,bromberg-rigidity} that
the complex dimension of the character variety~$X(\pi_1(M),\PSL)$ at a
discrete faithful representation is
equal to the number~$k$ of cusps of $M$.
In a small neighborhood of
the point $[\varrho_{y(u)}]\in X(\pi_1(M),\PSL)$, the chosen multicurve~$\theta$
defines local holomorphic coordinates via squared traces $\tr^2(\theta_l)$,
$1\leq l\leq k$; cf.~\cite{heusener-porti,craig-steve}.
In terms of $u$, we have $\tr^2(\theta_l) = 4\cosh^2\frac{u_l}{2}$, showing that these coordinates
diagonalize $D\paraX$.
We compute $\frac{\partial}{\partial u_l}\tr^2(\theta_l)=2\sinh(u_l)$, which
is non-zero whenever $0<|u_l|<\pi$.
Hence, $D\paraX$ is an isomorphism at these points.
\end{proof}
The well-known construction due to A.~Weil~\cite{weil-remarks}
identifies the tangent space to the space of conjugacy classes of representations,
here~$T_{[\varrho]}X(\pi_1(M),\PSL)$,
with the first cohomology group~$H^1(M;\Ad\rho)\cong H^1(M;\KKsub{M})$.
By applying Weil's isomorphism, we may therefore think of $D\paraX$
as taking values in $H^1(M;\KKsub{M})$.

\section{The cohomological content of the gluing equations}
\subsection{The long exact sequence associated to an ideal triangulation}
Consider $M$ with a hyperbolic structure obtained from a positively
oriented solution~$z\in\glvar$ and denote by $\KKsub{M}$ the sheaf of germs of
Killing vector fields on $M$.
Since $M_0$ is a closed subspace of $M$, we obtain a short exact sequence of sheaves
\(
	0\to\KKsub{M,M_0} \to \KKsub{M}\to \KKsub{M_0}\to 0
\),
whose associated long exact sequence in cohomology has the form
\begin{equation}\label{long-exact-cohomology}
	\begin{tikzcd}
	0\arrow{r} & H^0(M; \KKsub{M,M_0})\arrow{r} & H^0(M;\KKsub{M})\arrow{r} & H^0(M;\KKsub{M_0})
	\arrow[out=270, in=90, looseness=0.5]{dll} & \\
	 & H^1(M; \KKsub{M,M_0})\arrow{r} & H^1(M;\KKsub{M})\arrow{r} & H^1(M;\KKsub{M_0})
	\arrow[out=270, in=90, looseness=0.5]{dll} & \\
	 & H^2(M; \KKsub{M,M_0})\arrow{r} & H^2(M;\KKsub{M})\arrow{r} & H^2(M;\KKsub{M_0})\arrow{r} & 0.
	\end{tikzcd}
\end{equation}
Note that both $M$ and $M_0$ have the homotopy type of \nword{2}{complexes}, so
that we may terminate the sequence after the cohomology groups in degree~$2$.

\begin{lemma}[cf. Proposition~4.1.2 in \cite{rs-phd}]\label{vanishing-cohomology-groups}
	We have
	\[
		H^0(M; \KKsub{M,M_0}) =
		H^0(M;\KKsub{M}) =
		H^0(M;\KKsub{M_0}) =
		H^1(M; \KKsub{M,M_0}) =
		H^2(M;\KKsub{M_0}) = 0.
	\]
\end{lemma}
\begin{proof}
	Since $M$ has a discrete group of isometries, we have $H^0(M;\KKsub{M})=0$, which also implies the
	vanishing of $H^0(M;\KKsub{M,M_0})$.
	$H^1(M;\KKsub{M,M_0})$ vanishes since the sheaf~$\KKsub{M,M_0}$ is supported on the disjoint,
	contractible open sets~$E_i$.
	The isomorphism $H^2(M;\KKsub{M_0})=0$ follows from the fact that $M_0$ has the
	homotopy type of a graph.
	Finally, the cohomology group~$H^0(M;\KKsub{M_0})$ vanishes because the terms on either side
	of it in \eqref{long-exact-cohomology} have already been shown to vanish.
\end{proof}
\begin{corollary}
	For every $z\in\glvar$ and the corresponding hyperbolic structure on $M$,
	we have an exact sequence
	\begin{equation}\label{reduced-long-exact}
		0 \to H^1(M;\KKsub{M}) \to H^1(M;\KKsub{M_0}) \mapswith{\Delta} H^2(M;\KKsub{M,M_0})
		\to H^2(M;\KKsub{M}) \to 0.
	\end{equation}
\end{corollary}
\begin{remark}
Using the results of Porti~\cite{porti1997} and the foregoing discussion, it is
possible to determine the complex dimensions of the non-trivial terms in \eqref{reduced-long-exact}:
they equal $k$, $3N$, $3N$, and $k$, respectively.
Details can be found in \cite{rs-phd}; see~Lemma~3.4.1 and Section~4.1 in particular.
\end{remark}
The following theorem establishes a relationship between the infinitesimal
gluing equations~\eqref{Thurstons-consistency} and the exact sequence~\eqref{reduced-long-exact}.

\begin{theorem}\label{commutative-diagram-theorem}
	Let $\triang$ be an ideal triangulation of an open manifold $M$ in which all links of the ideal
	vertices are tori and let $\theta=\{\theta_1,\dotsc,\theta_k\}$ be a system of nontrivial
	oriented curves, one in	each vertex link.
	Assume that the gluing equations $\eqref{Thurstons-consistency}$ and $\eqref{Thurstons-completeness}$
	admit a positively oriented solution.
	Then there exist maps $\alpha$, $\beta$ and a neighborhood $U\subset\CC^k$ of the origin
	such that for any shapes $z$ with log-parameters $u=u(z)=(u_1,\dotsc,u_k)\in U$
	satisfying $0<|u_l|<\pi$ for all $l$,
	the following diagram is commutative with exact rows and columns.\\[0.5ex]
	\begin{tikzcd}[column sep=2em]
		& 0\arrow{d}
		& 0\arrow{d}
		& 0\arrow{d}
		& 0\arrow{d}
		\\
		  0\arrow{r}
		& T_{u(z)}U \arrow{r}{Dy} \arrow{d}{\textstyle D\paraX}
		& T_{z} \CC_{\Imag>0}^N \arrow{r}{Dg} \arrow{d}{\textstyle D\paraZ}
		& T_1(\CC^*)^N \arrow{r}{Dp} \arrow{d}{\textstyle\alpha}
		& \CC^k \arrow{r} \arrow{d}{\textstyle\beta}
		& 0\\
		0 \arrow{r}
		& H^1(M;\KKsub{M}) \arrow{r}\arrow{d}
		& H^1(M;\KKsub{M_0}) \arrow{r}{\Delta}\arrow{d}
		& H^2(M;\KKsub{M,M_0}) \arrow{r}\arrow{d}
		& H^2(M;\KKsub{M}) \arrow{r}\arrow{d}
		& 0
		\\
		& 0 \arrow{r}
		& \Coker D\paraZ \arrow{r}\arrow{d}
		& \Coker \alpha \arrow{r}\arrow{d}
		& 0
		\\
		&
		& 0
		& 0
		&
	\end{tikzcd}\\[0.3ex]
	Moreover, the maps~$\alpha$ and $\beta$ are unique.
\end{theorem}

The above theorem states in particular that we may view the derivative~$Dg$ of the consistency
equations~\eqref{Thurstons-consistency}
as the essential part of the connecting homomorphism~$\Delta$ in the cohomology long exact
sequence~\eqref{reduced-long-exact}.
We defer the proof until Section~\ref{proofsection-commut}.

\subsection{Cohomological meaning of complex lengths}
In this section, we explain the cohomological meaning of the dependence of
the completeness equations \eqref{Thurstons-completeness} on the chosen multicurve~$\theta$.
As before, we denote by $u=u(z)$ the log-parameter along $\theta$ and
consider the corresponding local
parametrization~$\paraX$ of Definition~\ref{def-paraZX}.

\begin{theorem}\label{log-parameters-cups}
For every $u=(u_1,\dotsc,u_k)\in U$ satisfying $0<|u_l|<\pi$ for all $l$, we have the
commutative diagram
\begin{equation*}
	\begin{tikzcd}
	& T_uU \arrow{r}{D\paraX}\arrow{d}{\cong}
	& H^1(M;\KKsub{M})\arrow{r}
	& H^1(M;\KKsub{\partial\overline{M}})\arrow{d}{\displaystyle\smile[\theta]^*}\\
	& \CC^k \arrow{r}{\beta}
	& H^2(M;\KKsub{M})\arrow{r}
	& H^2(M;\KKsub{\partial\overline{M}}),
	\end{tikzcd}
\end{equation*}
where $\beta$ is the map of Theorem~$\ref{commutative-diagram-theorem}$.
In the above diagram, $[\theta]^*\in H^1(\partial\overline{M},\ZZ)$ denotes the Poincar\'e
dual of the homology class of $\theta$.
\end{theorem}
\noindent
The above theorem is proved in Section~\ref{proof-of-cups}.

\subsection{Construction of the embedding}\label{sec-alpha}
At present, we are going to construct the unique map $\alpha$ which makes
Theorem~\ref{commutative-diagram-theorem} hold.

\begin{definition}\label{def-t}
	\begin{enumerate}[(i)]
		\item
			For any connected orientable manifold~$M$,
			we denote by $\Or(M)$ the set consisting of the
			two possible orientations of $M$.
		\item
			Let $L$ be a simple geodesic in a hyperbolic \nword{3}{manifold}~$M$ and
			let $\nu(L)\subset M$ be an
			open tubular neighborhood of $L$.
			For any orientation $\varepsilon\in\Or(L)$ of $L$, we denote by
			$t(L,\varepsilon)\in\KK(\nu(L))$ the unique local Killing vector field
			which acts as a unit-speed infinitesimal translation along $L$ in the
			direction of $\varepsilon$.
	\end{enumerate}
\end{definition}
\begin{remark}\label{oppo-orient}
The existence and uniqueness
of $t(L,\varepsilon)$ both follow from the results of Nomizu~\cite[\S2]{nomizu},
who worked in the more general setting of symmetric Riemannian manifolds.
It is easy to see that $t(L,-\varepsilon)=-t(L,\varepsilon)$, where $-\varepsilon$
is the orientation opposite to $\varepsilon$.
\end{remark}

\begin{example}\label{example-zero-infinity}
Consider the geodesic~$L\subset\HH$ connecting the points
$0,\infty\in\CC{}P^1=\partial_\infty\HH$ and let $\varepsilon$ be the orientation of $L$
from $0$ towards $\infty$.
Although $t(L,\varepsilon)$ is defined a priori only on a neighborhood of $L$, the fact that $\HH$
is simply connected allows us to continue $t(L,\varepsilon)$, as a Killing field,
unambiguously onto all of $\HH$.
Using the identification $\KK(\HH)=\slC$, we can write $t(L,\varepsilon)$ as a traceless $2\times2$ matrix
with complex entries.
To find this matrix, observe that for $s\in\RR$,
the M\"obius transformation $z\mapsto e^sz$ acts on $L$ as a translation
by $s$ units towards $\infty$.
Hence,
\begin{equation}\label{calculation-of-t}
	t(L,\varepsilon) = \frac{d}{ds} \begin{bmatrix}e^{s/2} & 0\\ 0 & e^{-s/2}\end{bmatrix}_{s=0}
	= \frac{1}{2}\begin{bmatrix} 1 & 0 \\ 0 & -1 \end{bmatrix} \in\slC.
\end{equation}
\end{example}

For all edges~$e_i$ of the ideal triangulation~$\triang$,
$1\leq i \leq N$, consider the neighborhoods $E_i$ of Definition~\ref{def-M0}.
Since the sheaf $\KKsub{M,M_0}$ is supported on the disjoint open sets $E_i$, its second
cohomology group splits naturally as
\[
	H^2(M;\KKsub{M,M_0}) \cong \prod_{i=1}^N H^2(M;\KKsub{M,M_i}),
	\ \text{where}\ M_i = M\setminus E_i.
\]
By excision, to calculate the cohomology groups on the right-hand side it suffices to consider,
for each $i$, an embedded disc~$D_i$ transverse to $e_i$ and satisfying
$D_i\cap E_i = \interior D_i$.
Using cellular cochains with $\interior D_i$ as a \nword2cell, we immediately see
that any element of $H^2(M;\KKsub{M,M_i})$
is fully determined by its value on the oriented disc~$D_i$.
In other words, given an orientation $d_i\in\Or(D_i)$, cohomology classes in $H^2(M;\KKsub{M,M_i})$
are in a one-to-one correspondence with
local sections of $\KK$ on $E_i$.
Observe that an edge orientation~$\varepsilon_i\in\Or(e_i)$
determines a dual orientation $d_i\in\Or(D_i)$ by the requirement
that $d_i\wedge\varepsilon_i$ agrees with the orientation of the ambient manifold $M$.
Let $\varepsilon=(\varepsilon_1,\dotsc,\varepsilon_N)\in\prod_i\Or(e_i)$ be an arbitrary choice of orientations of the
edges of $\triang$. Applying the foregoing reasoning to all edges of $\triang$,
we obtain the isomorphism
\begin{equation}\label{disc-isomorphism}
	\auxiliaryMap_\varepsilon:\prod_{i=1}^N \KK(E_i)\mapswith{\cong} H^2(M;\KKsub{M,M_0}).
\end{equation}

We write $x_1,\dotsc,x_N$ for the coordinates on $(\CC^*)^N$; hence, the complex
tangent space $T_1(\CC^*)^N$ is spanned by the vectors
$\partial/\partial x_1,\dotsc,\partial/\partial x_N$.
Define the map $\omega_\varepsilon$ by
\begin{equation}\label{x-to-t-isomorphism}
	\omega_\varepsilon: T_1(\CC^*)^N\to \prod_{i=1}^N \KK(E_i), \qquad
	\omega_\varepsilon\left(\frac{\partial}{\partial x_i}\right)=t(e_i, \varepsilon_i)\ \text{for all}\
	1\leq i\leq N.
\end{equation}
We set $\alpha = \auxiliaryMap_\varepsilon \circ \omega_\varepsilon$.
By Remark~\ref{oppo-orient}, $\alpha$ does not depend on $\varepsilon$.
We claim that this is the map needed in Theorem~\ref{commutative-diagram-theorem}.

\subsection{Proof of Theorem~\ref{commutative-diagram-theorem}}\label{proofsection-commut}
We are now going to prove
Theorem~\ref{commutative-diagram-theorem};
we refer to \cite[\S4.3]{rs-phd} for more details and illustrations.

\begin{proof}[Proof of Theorem~\ref{commutative-diagram-theorem}]
Observe that the leftmost square in the top part of the diagram is commutative by
the definition of $\paraX$.
We shall now prove the commutativity of the central square, ie, the equality
$\alpha\circ Dg=\Delta\circ D\paraZ$.

We fix edge orientations $\varepsilon=(\varepsilon_1,\dotsc,\varepsilon_N)\in\prod_i \Or(e_i)$ arbitrarily.
As discussed in the preceding section, $\varepsilon_i$ determines a dual orientation of a transverse
disc $D_i$ for every $i$.
Denote by $\gamma_i$ the oriented boundary of $D_i$.
There exists a local orientation-preserving coordinate chart which identifies $e_i$ with
the oriented infinite geodesic $L=(0,\infty)\subset\HH$ of Example~\ref{example-zero-infinity}.
Any such chart also identifies the space of Killing fields on $E_i$ with $\KK(\HH)=\slC$.
In these coordinates, the monodromy of the hyperbolic structure
along $\gamma_i$ is a homothety~$h_s(w)=sw$ whose ratio~$s=g_i(z)$
is the product of shape parameters of the tetrahedra incident to $e_i$;
in particular, $s=1$ when $z=z_0\in\glvar$.
We choose a local logarithm~$\ell_i(z)$ so that $g_i(z)=\exp(\ell_i(z))$ and
$\ell_i(z_0)=0$ for a given $z_0\in\glvar$.
Then for any $z$ sufficiently close to $z_0$, the monodromy along $\gamma_i$ can be written in
matrix form as
\[
	\mu_i(z)=\pm
	\begin{bmatrix}
		\exp(\ell_i(z)/2) & 0 \\
		0 & \exp(-\ell_i(z)/2)
	\end{bmatrix}\in\PSL.
\]
Using Weil's method~\cite{weil-remarks} and performing a calculation similar
to~\eqref{calculation-of-t}, we see that the infinitesimal variation of the monodromy
along $\gamma_i$ with respect to $z_j$ is given by
\begin{equation}\label{Weil-differentiation}
	\frac{\partial}{\partial z_j} \mu_i(z)\Bigr|_{z=z_0}\mu_i^{-1}(z_0)
	= \frac{1}{2} \frac{\partial\ell_i(z)}{\partial z_j}\Bigr|_{z=z_0}
			\begin{bmatrix} 1 & 0 \\ 0 & -1\end{bmatrix}
	= \frac{\partial g_i(z)}{\partial z_j}\Bigr|_{z=z_0} t(e_i, \varepsilon_i),
\end{equation}
where the last equality uses $g_i(z_0)=1$.
Using the isomorphism~\eqref{disc-isomorphism} and the fact that $\gamma_i=\partial D_i$,
we see that the right-hand side of~\eqref{Weil-differentiation} is
the $i$-th component of
$\auxiliaryMap_\varepsilon^{-1}\left(\Delta\bigl(D\paraZ(\partial/\partial z_j)\bigr)\right)$.
On the other hand, using \eqref{x-to-t-isomorphism} we get
\[
	\auxiliaryMap_\varepsilon^{-1}\left(\alpha\bigl(Dg(\partial/\partial z_j)\bigr)\right)
	= \omega_\varepsilon \left(\sum_{i=1}^N \frac{\partial g_i(z)}{\partial z_j}
			\frac{\partial}{\partial x_i}\right)
	= \sum_{i=1}^N \frac{\partial g_i(z)}{\partial z_j} t(e_i,\varepsilon_i).
\]
At $z=z_0$, the $i$-th term of the above sum is exactly the right-hand
side of \eqref{Weil-differentiation}.
Hence, the middle square commutes.

Using surjectivity of $Dp$ and a standard diagram chase,
we may now define $\beta$ to be the unique map for which the rightmost square
commutes.
In this way, all squares in the top part have been shown commutative.

By Lemma~\ref{injectivity-lemma}, the map $D\paraX$ is an isomorphism.
Moreover, $\alpha$ is injective by definition.
The Four Lemma now implies that $D\paraZ$ is injective as well.

It remains to be shown that $\beta$ is an isomorphism.
We start by taking the quotient of the second row by the image of the first row,
which results in the following commutative diagram:
\begin{equation}\label{cokernel-complex-8-lemma}
	\begin{tikzcd}
	&T_{z} \CC_{\Imag>0}^N \arrow{r}{Dg} \arrow[hook]{d}{\textstyle D\paraZ}
	& T_1(\CC^*)^N \arrow[hook]{d}{\alpha}\arrow[two heads]{r}{Dp}
	& \CC^k\arrow{r}\arrow{d}{\beta}
	& 0\\
	&H^1(M;\KKsub{M_0})\arrow{r}{\Delta}\arrow[two heads]{d}
	& H^2(M;\KKsub{M,M_0})\arrow{r}\arrow[two heads]{d}
	& H^2(M;\KKsub{M})\arrow[two heads]{d}\arrow{r}
	& 0\\
	0\arrow{r}
	& \Coker D\paraZ\arrow{r}
	& \Coker\alpha\arrow{r}
	& \Coker\beta\arrow{r}
	& 0.
	\end{tikzcd}
\end{equation}
Since the bottom row of \eqref{cokernel-complex-8-lemma} is exact,
the map $\Coker D\paraZ\to\Coker\alpha$ is injective.
In order to see that this map is an isomorphism,
it suffices to compute the dimensions of its domain and codomain:
\begin{align*}
	\dim_\CC\bigl(\Coker D\paraZ\bigr)
		=\dim_\CC\bigl(H^1(M;\KKsub{M_0})\bigr) - N
		&= 2N,\\
	\dim_\CC\bigl(\Coker\alpha\bigr)
		=\dim_\CC\bigl(H^2(M;\KKsub{M,M_0})\bigr)-N
		&= 2N.
\end{align*}
This implies that $\Coker\beta=0$, i.e., that $\beta$ is surjective.
Since $\dim_\CC H^2(M;\KKsub{M})=k$, $\beta$ must in fact be an isomorphism.
\end{proof}

\subsection{Complex lengths of peripheral curves}\label{proof-of-cups}
The boundary at infinity $\partial\overline{M}$ can be pushed
into $M$, yielding a disjoint union of embedded tori~$T_1\cup\dotsb\cup T_k$,
numbered according to the chosen numbering of the ends of $M$.
When $M$ is equipped with either the complete hyperbolic structure or a small deformation of it,
we have the natural isomorphisms
\[
	H^2(M;\KKsub{M})
	\cong H^2(\partial\overline{M};\KKsub{\partial\overline{M}})
	\cong \prod_{l=1}^k H^2(T_l;\KKsub{T_l});
\]
the first isomorphism follows from the $n=3$~case of Theorem~0.1 in
\cite{porti-ferrer} and is also stated as Proposition~3.4.2 in \cite{rs-phd}.
It turns out that in the incomplete case, a basis of $H^2(T_l;\KKsub{T_l})$
can be constructed geometrically.

\begin{lemma}[cf. {\cite[Proposition~4.4.1]{rs-phd}}]\label{ts-span}
Let $T$ be a torus about an incomplete end~$v$ of an oriented hyperbolic \nword{3}{manifold} $M$ and
let $R$ be a geodesic ray traveling into $v$ with the orientation~$\varepsilon\in\Or(R)$ pointing
towards $v$.
Equip $T$ with a cell decomposition containing a single \nword2cell~$S$ and orient $S$ positively
using the orientation of $M$.
Then the \nword2cochain mapping $S$ to $t(R,\varepsilon)$ defines a non-trivial cohomology
class~$[t]\in H^2(T;\KKsub{T})$.
Moreover, $[t]$ does not depend on the choice of the ray~$R$.
\end{lemma}
\begin{proof}
Since the projective structure on $T$ reduces to an affine
structure \cite[Lemma~2]{hubbard-monodromy},
the local system $\KKsub{T}$ can be understood in terms of the adjoint of the monodromy representation
$\pi_1(T)\to\Aff(\CC)$, where $\Aff(\CC)$ is embedded into $\PSL$ as the upper-triangular Borel
subgroup.
As in Example~\ref{example-zero-infinity}, $t(R,\varepsilon)$ is then written as the traceless
diagonal matrix $\left[\begin{smallmatrix}1/2 & 0\\ 0 & -1/2\end{smallmatrix}\right]\in\slC$.
This reduces the proof to an elementary computation, which can be found in~\cite[pp.~94-95]{rs-phd}.
\end{proof}

Suppose that $T$ is a torus about an incomplete end of $M$ and that $\gamma\subset T$ is an
oriented simple closed curve representing a non-trivial free homotopy class in $T$.
Then the monodromy~$\mu(\gamma)$ of the hyperbolic structure along $\gamma$
can be conjugated into the form
\begin{equation}\label{upper-triangularized-affine}
	\mu(\gamma) = \pm\begin{bmatrix} e^{\Len/2} & *\\ 0 & e^{-\Len/2}\end{bmatrix},
\end{equation}
where $\Len$ is called the \emph{complex length} of $\gamma$.
Assuming $\Len\not\in2\pi i\ZZ$, Bromberg~\cite[p.~25]{bromberg-rigidity} defines
an isomorphism
\begin{equation}\label{brombergs-iso}
	B: H^1(\gamma; \KKsub{\gamma})\mapswith{\cong}\CC
\end{equation}
which sends a cohomology class to the corresponding infinitesimal variation of $\Len$.
Observe that $\Len$ is not defined uniquely by $\mu(\gamma)$ alone:
even after choosing a branch of the logarithm, swapping the eigenvalues
will replace $\Len$ with $-\Len+\const$.
Hence, the derivative of $\Len$ is defined \emph{a priori} only up to sign.
On the other hand, when $\gamma$ is in normal position
with respect to a positively oriented geometric triangulation~$\triang$,
the log-parameter~$u=u_\gamma$ provides a particular choice of $\Len$.
Hence, by setting $\Len=u$ locally, we obtain a convenient choice of the isomorphism~$B$
which we characterize below.

\begin{lemma}\label{little-commutative}
If $\gamma\subset T$ is an oriented curve as above, we have a commutative diagram
\[
	\begin{tikzcd}
		H^1(T;\KKsub{T}) \arrow{r}{\textstyle\smile[\gamma]^*}\arrow{d}
		 & H^2(T;\KKsub{T})\arrow{d}{\textstyle [t]\mapsto 1}\\
		H^1(\gamma;\KKsub{\gamma})\arrow{r}{B} & \CC,
	\end{tikzcd}
\]
where $[\gamma]^*\in H^1(T;\ZZ)$ is the Poincar\'e dual of the homology class of $\gamma$ and $[t]$
is the element constructed in Lemma~\ref{ts-span}.
\end{lemma}
\begin{proof}
We can choose a fundamental quadrilateral~$Q\subset\CC$ for $T$ in such a way that
one of the sides of $Q$ is a lift of $\gamma$ with the basepoint at one of the vertices.
Moreover, we may place the basepoint vertex of $Q$ at $0\in\CC$.
This conjugates the monodromy representation $\mu:\pi_1(T)\to\Aff(\CC)$
into the form \eqref{upper-triangularized-affine}.
If $L\subset M$ is a hyperbolic geodesic intersecting $T$ orthogonally at the basepoint,
then it is clear that $\mu(\gamma)$ acts on $L$ by a translation of $\Real\Len$ and rotation
through the angle $\Imag\Len$.
In other words, if $L$ is taken to the line
$(0,\infty)\subset\HH$ in the upper-halfspace model, $\mu(\gamma)$ acts as
the M\"obius transformation $z\mapsto e^\Len z$.

Suppose now that a cohomology class $h\in H^1(T;\KKsub{T})$ is tangent to a holomorphic \nword{1}{parameter}
family of projective structures with monodromy $\mu_s(\gamma)=(z\mapsto e^{\Len(s)}z+c(s))$,
so that $\Len=\Len(0)$.
Denote by $R$ the ray formed by points of $L$ on the thin side of $T$.
Using Lemma~\ref{ts-span} with this choice of $R$, we see that
$h\smile[\gamma]^* = \frac{d\Len(s)}{ds}\big|_{s=0}[t]$ and the result follows.
\end{proof}

\begin{proof}[Proof of Theorem~\ref{log-parameters-cups}]
Observe that Theorem~\ref{log-parameters-cups} follows easily from
Lemma~\ref{little-commutative} once we prove that the composition
\[
	\CC^k \mapswith{\beta} H^2(M;\KKsub{M}) \to H^2(M;\KKsub{\partial\overline{M}})
\]
sends the $l$-th standard unit vector of $\CC^k$ to the cohomology
class~$[t_l]\in H^2(T_l;\KKsub{T_l})$ constructed in Lemma~\ref{ts-span} on the $l$-th boundary
torus~$T_l$.
To see that this is indeed the case, observe that the Jacobian matrix of the monomial map~$p$
of \eqref{def-p} is $K=[K_{li}]$, where
$K_{li}$ was defined in Section~\ref{gluing-section} as the unsigned number
of ends of the edge $e_i$ incident to the $l$-th end of $M$.
Hence, using the commutative diagram in Theorem~\ref{commutative-diagram-theorem},
it suffices to show that for every $i$, the basis vector
$\frac{\partial}{\partial x_i}\in T_1(\CC^*)^N$
is sent by the composition
\begin{equation*}\label{edge-to-torus}
	r_l: T_1(\CC^*)^N \mapswith{\alpha} H^2(M;\KKsub{M,M_0}) \to H^2(M;\KKsub{M}) \to H^2(T_l;\KKsub{T_l})
\end{equation*}
to $K_{li}[t_l]$.
To compute $r_l(\dby{x_i})$, split the edge $e_i$ into two rays $R_1$ and $R_2$ with orientations
$\varepsilon_1$, $\varepsilon_2$ pointing outwards (towards infinity).
Then $\omega_{\varepsilon_n}(\frac{\partial}{\partial x_i})=t(R_n,\varepsilon_n)$ for $n=1,2$.
Hence, if $v_l$ is the ideal vertex of $\triang$ whose link is $T_l$, then
\[
	r_l\bigl({\textstyle\frac{\partial}{\partial x_i}}\bigr)
	= \sum_{
		\substack{\text{ends of}\ e_i\\
		\text{incident to}\ v_l}
	}\mkern-10mu[t_l]
	= K_{li}[t_l],
\]
as desired.
\end{proof}

\section{Combinatorial Reidemeister torsion of 3-manifolds}
\subsection{Review of algebraic torsion}\label{torsion-review}
In this section, we briefly summarize the definition of the algebraic torsion
of a cochain complex, while referring the reader to \cite{derham-torsion,turaev}
for more details.

Suppose that $V$ is a vector space of finite dimension $n$ over a field $\FF$.
Any ordered basis $\basis{b}=(b_1,\dotsc,b_n)$ of $V$ determines a non-zero vector
$\vol(\basis{b}):=\bigwedge_{r=1}^{n} b_r\in\bigwedge^n V$
in the top-degree exterior power of $V$.
When the elements of $\basis{b}$ are not ordered, $\vol(\basis{b})$
is only defined up to sign.
Similarly, when $\basis{b'}\subset V$ is another (unordered) basis, the
ratio~$\vol(\basis{b})\big/\vol(\basis{b'})$ is a scalar well defined up to sign,
which can be computed in practice as the determinant of a change-of-basis matrix from
$\basis{b}$ to $\basis{b'}$.

Let \(
	C^\bullet =
	(0 \to C^0 \mapswith{\delta^0} C^{1} \mapswith{\delta^1} \dotsb
	\mapswith{\delta^{d-2}} C^{d-1} \mapswith{\delta^{d-1}} C^d \to 0)
\)
be a finite-dimensional cochain complex over $\FF$
in which each cochain group $C^i$ is equipped with a preferred basis $\basis{c^i}$.
Since we are working over a field, the short exact sequence
\[
	0 \to Z^i \mapswith{\subseteq} C^i \mapswith{\delta^{i}} B^{i+1} \to 0
\]
always has a splitting $s^i: B^{i+1}\to C^i$.
Denote by $\basis{h^i}\subset H^i=Z^i/B^i$ an arbitrarily fixed basis of
the $i$-th cohomology group of $C^\bullet$.
Note that given any collection of bases $\basis{b^i}\subset B^i$ for each $i$,
we can form a new basis of $C^i$ defined as
$\basis{b^i}\cup\basis{\widetilde{h^i}}\cup s^i(\basis{b^{i+1}})$, where
$\basis{\widetilde{h^i}}$ consists of cocycles representing the cohomology classes of
the elements of $\basis{h^i}$.
The combinatorial torsion of $C^\bullet$ is then defined as
\begin{equation}\label{tordef}
	\Tor(C^\bullet, \basis{c^\bullet}, \basis{h^\bullet})
	= \pm\prod_{i\ \text{even}}
		\frac{\vol\bigl(\basis{b^i}\cup\basis{\widetilde{h^i}}\cup s^i(\basis{b^{i+1}})\bigr)}
			{\vol\bigl(\basis{c^i}\bigr)}
	\prod_{i\ \text{odd}}
		\frac{\vol\bigl(\basis{c^i}\bigr)}
		{\vol\bigl(\basis{b^i}\cup\basis{\widetilde{h^i}}\cup s^i(\basis{b^{i+1}})\bigr)}
	\in\FF^*/\{\pm 1\},
\end{equation}
and only depends on the choice of the bases $\basis{c^\bullet}$ and $\basis{h^\bullet}$.

When the above construction is applied to a cellular cochain complex of a topological space, the
resulting combinatorial invariant of cell complexes
is called the \emph{combinatorial Reidemeister torsion}.
For further generalization to modules over non-commutative rings,
see Milnor~\cite{milnor-whitehead}.

We shall also need the notion of compatible bases introduced in \cite{milnor-whitehead}.
Suppose that
\begin{equation}\label{basic-ses}
	0 \to A \mapswith{\iota}  B \mapswith{\pi} C \to 0
\end{equation}
is a short exact sequence of finite-dimensional vector spaces.
We say that the bases~$\basis{a}\subset A$, $\basis{b}\subset B$, $\basis{c}\subset C$
are \emph{compatible} if the torsion of \eqref{basic-ses} with respect to these
bases equals $\pm1$.
More generally, if $A$, $B$ and $C$ are cochain complexes and $\iota$, $\pi$ cochain maps,
we say that the bases~$\basis{a}$, $\basis{b}$, $\basis{c}$ are \emph{graded-compatible}
if their degree $d$ parts~$\basis{a}^d$, $\basis{b}^d$, $\basis{c}^d$ are compatible
for every $d$.

\subsection{Definition of the adjoint torsion}
The adjoint hyperbolic torsion $\ATor$ was first defined by Porti in~\cite{porti1997}.
While Porti's original treatment was in terms of homology groups, our approach
using cohomology is equivalent \cite{turaev}.

Let $M$ be as in Section~\ref{gluing-section} and
let $\varrho$ be a monodromy representation of the hyperbolic structure on $M$.
We equip $M$ with an arbitrary finite CW-decomposition and consider the finite-dimensional
cellular cochain complex $C^\bullet(M;E)\cong C^\bullet(M; \slC^{\Ad \varrho})$ with twisted
\nword{\slC}{coefficients}.
This complex can be constructed as $\Hom_{\ZZ[\pi_1(M)]}(C^\bullet(\widetilde{M};\ZZ),\slC)$,
where $\pi_1(M)$ acts on the universal covering space $\widetilde{M}$ by deck transformations
and on $\slC$ via $\Ad\varrho$.
As stated in \cite{porti1997,rs-phd}, the cohomology groups of this complex are
\[
	H^1(M; \slC^{\Ad \varrho})\cong H^2(M; \slC^{\Ad \varrho}) \cong \CC^k
\]
and vanish in all other degrees.

In order to take the combinatorial torsion of $C^\bullet(M; \slC^{\Ad \varrho})$, we
equip it with a \emph{geometric basis} which can be constructed as follows.
Let $\basis{b}=\{b_1,b_2,b_3\}\subset\slC$ be an arbitrarily chosen basis.
The cellular structure of $M$ determines a $\pi_1(M)$-invariant cell decomposition of $\widetilde{M}$.
For any oriented cell~$s$ in $M$, choose a lift~$\widetilde{s}$ of $s$
to $\widetilde{M}$ and form the three cochains $c_{s,r}$ ($r=1,2,3$) defined by
$c_{s,r}(\widetilde{s})=b_r$ and $c_{s,r}(f)=0$ if the cell~$f$ is not a lift of $s$.
Then define
\begin{equation}\label{geometric-basis}
	\basis{c}^\bullet_{\text{geom}}
	= \bigcup_{s:\ \text{a cell of}\ M} \mkern-10mu \{c_{s,1},c_{s,2}, c_{s,3}\}
	\subset C^\bullet(M; \slC^{\Ad \varrho}),
\end{equation}
which is easily seen to be a basis.

\begin{definition}\label{Ad-torsion}
Let $\gamma=(\gamma_1,\dotsc,\gamma_k)\subset \partial\overline{M}$ be a multicurve consisting of
oriented, homotopically non-trivial simple closed curves, one in each torus component of
$\partial\overline{M}$.
\begin{enumerate}[(i)]
\item\label{balanced}
A cohomology basis
$\basis{h^1}\cup \basis{h^2}\subset H^1(M; \slC^{\Ad \varrho})\oplus H^2(M; \slC^{\Ad \varrho})$
is said to be \emph{balanced with respect to $\gamma$} if the images of $\basis{h^1}$ and
$\basis{h^2}$ under the compositions
\[
	\begin{tikzcd}
	\basis{h^1} \subset H^1(M; \slC^{\Ad \varrho}) \arrow{r}{}
	& H^1(\partial M; \slC^{\Ad \varrho}) \arrow{d}{\textstyle\smile[\gamma]^*}\\
	\basis{h^2} \subset H^2(M; \slC^{\Ad \varrho}) \arrow{r}{}
	& H^2(\partial M; \slC^{\Ad \varrho})
	\end{tikzcd}
\]
coincide.
In the above diagram, the horizontal maps are induced by restriction of local systems
and the vertical map is given by the cup product with the Poincar\'e
dual~$[\gamma]^*\in H^1(\partial M; \ZZ)$ of the homology class of $\gamma$.
\item
The \emph{adjoint Reidemeister torsion} of $(M,\gamma)$ is defined by
\[
	\ATor(M, \gamma) := \Tor\bigl(C^\bullet(M; \slC^{\Ad \varrho}),
				\basis{c}^\bullet_{\text{\rm geom}},
				\basis{h^\bullet}(\gamma)\bigr)
			\in\CC^*/\{\pm1\},
\]
where $\basis{c}^\bullet_{\text{geom}}$ is any geometric basis constructed by the way of
\eqref{geometric-basis} and $\basis{h^\bullet}(\gamma)$ is any cohomology basis balanced with
respect to $\gamma$.
\end{enumerate}
\end{definition}
It is well known that the quantity $\ATor(M, \gamma)$ does not depend on the choice of a geometric
basis; cf.~\cite{porti1997}.
Note that Part~\eqref{balanced} of the above definition is an adaptation of Porti's original
construction to twisted cochain complexes.

\subsection{Geometric construction of geometric bases}\label{geom-bases}
We shall now indicate how to reformulate Definition~\ref{Ad-torsion} in terms of the sheaf~$\KKsub{M}$
of germs of Killing vector fields on $M$.
As mentioned in the Introduction, the local system $\slC^{\Ad\varrho}$ is isomorphic
to the sheaf~$\KKsub{M}$.
Hence, Part~\eqref{balanced} of Definition~\ref{Ad-torsion} does not require any adaptations
beyond replacing ``$\slC^{\Ad\varrho}$'' with ``$\KKsub{M}$'' throughout.
The definition of a geometric basis \eqref{geometric-basis}
can also be restated in a simple way which we now describe.

Given an ordered basis $\basis{b}=(b_1,b_2,b_3)\subset\slC$, \eqref{explicit-E}
implies that the non-zero element $\vol(\basis{b})\in\bigwedge^3\slC$
defines a non-vanishing section of the bundle $\bigwedge^3E$.
Since the adjoint action of $\PSL$ is unimodular, $\bigwedge^3E$ is trivial as a flat bundle
and $\vol(\basis{b})$ can be viewed as a global, constant section of $\bigwedge^3E$.

According to \eqref{geometric-basis}, a geometric basis contains, for every oriented
cell~$s$, three cochains $c_{s,r}$ ($r=1,2,3$) such that
$\bigwedge_r c_{s,r}(\widetilde{s})=\vol(\basis{b})$
as sections of $\bigwedge^3E$.
Using Steenrod's definition~\cite{steenrod1943} of cellular cochains
with coefficients in a local system, we can interpret this equality in terms of
the cellular cochain complex~$C^\bullet(M;\KKsub{M})$.
For every oriented cell~$s$, consider three germs~$X_r\in\KKsub{x}$ ($r=1,2,3$)
such that $\bigwedge_r X_r = \vol(\basis{b})_x$, where $x\in s$ is an arbitrarily
chosen ``representative point''.
It is clear that replacing $c_{s,r}$ with cochains~$c'_{s,r}$ mapping~$s$ to $X_r$
for every $s$ produces a basis~$\basis{c'}\subset C^\bullet(M;\KKsub{M})$.
For every degree~$d$, the identification of local systems~$\slC^{\Ad\varrho}\cong\KKsub{M}$
induces an isomorphism of top exterior
powers $\bigwedge C^d(M;\slC^{\Ad\varrho})\mapswith{\cong}\bigwedge C^d(M;\KKsub{M})$.
It is easy to see that this isomorphism relates $\basis{c}^d_{\text{geom}}$ to the
degree~$d$ part of $\basis{c'}$.
Hence, we have the equality of torsions
$\Tor\bigl(C^\bullet(M;\KKsub{M}), \basis{c'}, \basis{h}\bigr) =
\Tor\bigl(C^\bullet(M;\slC^{\Ad\varrho}), \basis{c}^\bullet_{\text{geom}}, \basis{h}\bigr)$.

In order to make the construction of the germs~$X_r$ more geometric,
fix a cell~$s$ of $M$ and a representative point $x \in s$.
Let $U\subset M$ be a neighborhood of $x$ and let $\varphi: U\to\HH$
be an orientation-preserving geometric coordinate chart.
Then $\varphi$ establishes an isomorphism~$\KKsub{x}\cong\KKsub{\varphi(x)}$ of germ spaces and
hence an isomorphism of their top exterior powers $\bigwedge^3\KKsub{x} \cong \bigwedge^3\KKsub{\varphi(x)}$.
Since a germ of a Killing field at a point of $\HH$ extends to a unique global Killing field,
we obtain an isomorphism
\[
	\Phi: \bigwedge\nolimits^3\KKsub{x} \mapswith{\cong} \bigwedge\nolimits^3\slC.
\]
By unimodularity of the adjoint action of $\PSL=\Isom^+(\HH)$, we see that $\Phi$ does not
depend on the choice of the geometric chart $\varphi$.
Moreover, once $\Phi$ is defined at an~$x\in s$, it can be uniquely continued to any other point
of $s$.

In conclusion, we arrive at the following equivalent definition of a geometric basis of the
cellular cochain complex~$C^\bullet(M;\KKsub{M})$.
Fix a basis~$\basis{b}=\{b_1,b_2,b_3\}$ of $\slC\cong\KK(\HH)$.
For any oriented cell~$s$ of $M$ and any representative point $x\in s$,
choose an arbitrary orientation-preserving geometric chart $\varphi$ as above and define
\begin{equation}\label{really-geometric-basis}
	\basis{c}^\bullet_{\text{geom}}
	= \bigcup_{s:\ \text{a cell of}\ M} \bigcup_{r} \left\{s\mapsto (\varphi^{-1})_*(b_r)\right\}
	\subset C^\bullet(M; \KKsub{M}),
\end{equation}
where $(\cdot)_*$ denotes the push-forward of vector fields.
We remark that a similar coordinate-free interpretation of geometric bases exists for unimodular
flat bundles of any rank \cite{nicolaescu-book,rs-phd}.

\subsection{Generalized 1-loop Conjecture}
Assume for the entirety of this section that $\triang$ is a geometric ideal triangulation
of a hyperbolic \nword{3}{manifold} $M$ with one cusp ($k=1$).
Adopting notations of normal surface theory,
we shall often write $z_j^\square$ for the unique of the three shape
parameters $z_j$, $z'_j$, $z''_j$ labeling the edges of the $j$-th tetrahedron of $\triang$
which are separated by the normal quadrilateral $\square$.
We extend this notation to all quantities corresponding bijectively to
normal quadrilaterals.

\begin{definition}\label{flattening}
A \emph{strong combinatorial flattening} on $\triang$ is
a vector~$(f, f', f'')\in \bigl(\ZZ^{N}\bigr)^3$ satisfying the equations
\begin{align}
	Gf + G'f' + G''f'' &= (2, 2, \dotsc, 2)^\top,\\
	f + f' + f'' &= (1,1, \dotsc, 1)^\top,\\
	Cf + C'f' + C''f'' &= 0,\label{peripheral-flattening-condition}
\end{align}
whenever the rows of the matrices $C^\square$ contain the coefficients of the completeness
equations~\eqref{Thurstons-completeness} along any nontrivial peripheral curves.
\end{definition}

We remark that strong combinatorial flattenings always exist, thanks to a result by
Neumann~\cite[Theorem~2]{neumann1990}.
By linearity, it suffices to check condition \eqref{peripheral-flattening-condition}
on a \nword{\ZZ}{basis} of $H_1(\partial\overline{M};\ZZ)$.
If $M$ is a knot complement in $S^3$, one may use the
knot-theoretic meridian--longitude pair, as is done
by Dimofte and Garoufalidis in \cite{tudor-stavros}.
Note that \cite{tudor-stavros} uses the weaker term ``flattening'' to refer to
a vector~$f$ satisfying \eqref{peripheral-flattening-condition} only for a single peripheral curve,
and the term ``flattening compatible with longitude'' to describe strong combinatorial flattenings
in our sense.

The `\nword{1}{loop} Conjecture' of Dimofte--Garoufalidis~\cite[Conjecture~1.8]{tudor-stavros}
can be stated using the notations of Section~\ref{gluing-section} as follows.
Let $\theta$ be an oriented, homotopically nontrivial simple closed peripheral curve
in normal position with respect to the triangulation~$\triang$.
Define the $N\times N$ integer matrices $\hat{G}, \hat{G}', \hat{G}''$ by
\begin{equation}\label{def-hats}
	\hat{G}^\square_{ij} = \begin{cases}
	G^\square_{ij} & \text{when}\ 1\leq i \leq N-k\\
	C^\square_{i-N+k,j} & \text{when}\ i>N-k,
	\end{cases}
\end{equation}
where $C^\square$ contains the coefficients of the completeness
equation~\eqref{Thurstons-completeness} along $\theta$.
Since we are only considering the case of~$k=1$ here,
the matrices~$\hat{G}^\square$ differ from $G^\square$ only in their last
rows, but see Remark~\ref{long-remark}\nobreakdash-\eqref{case:k>1} below.
For any $j\in\{1,\dotsc,N\}$, we define the following rational functions of $z\in\CC_{\Imag>0}^N$:
\begin{equation}\label{def-zetas}
	\zeta_j(z) = \frac{d\log z_j}{dz_j} = \frac{1}{z_j},\quad
	\zeta'_j(z) = \frac{d\log z'_j}{dz_j}= \frac{1}{1-z_j},\quad
	\zeta''_j(z) = \frac{d\log z''_j}{dz_j} = \frac{1}{z_j(z_j-1)}.
\end{equation}

\begin{conjecture}[The 1-loop Conjecture]\label{1-loop}
For any strong combinatorial flattening $(f, f', f'')$,
\begin{equation}\label{1-loop-formula}
	\ATor(M, \theta) = \pm\frac{\det\bigl(
			\hat{G}\diag(\zeta) + \hat{G}'\diag(\zeta') + \hat{G}''\diag(\zeta'')
			\bigr)}{2^k\;{\zeta}^{f}{\zeta'}^{f'}{\zeta''}^{f''}}.
\end{equation}
\end{conjecture}
In the formula \eqref{1-loop-formula}, the manifold $M$ is considered
with the hyperbolic structure defined by the shape parameters $z=(z_1,\dotsc,z_N)\in\glvar$
which enter the right-hand side via the functions $\zeta^\square(z)$ of \eqref{def-zetas}.
The denominator uses multi-index notation.

\begin{remark}\label{long-remark}
	\begin{enumerate}[(1)]
		\item
		We remark that our statement of the \nword1loop Conjecture extends the original
		conjecture of Dimofte--Garoufalidis~\cite{tudor-stavros} to the case of arbitrary
		\nword1cusped hyperbolic manifolds and arbitrary nontrivial peripheral curves.
 		\item
		The conjectural formula in \cite{tudor-stavros} is given in terms of the matrices
		$\mathbf{A}=\hat{G}-\hat{G}'$ and $\mathbf{B}=\hat{G}''-\hat{G}'$, but it can be
		easily transformed into the symmetric form presented above.
		The details of this transformation are spelled out in \cite[\S5.1.3]{rs-phd}.
		\item
		Dimofte--Garoufalidis~\cite{tudor-stavros} use a somewhat weaker
		concept of a combinatorial flattening and prove that
		the right-hand side of \eqref{1-loop-formula} does not depend on the choice of
		the flattening when the shapes $z$ recover the complete hyperbolic structure of
		finite volume on $M$, but subsequently need strong combinatorial flattenings for
		full invariance.
		Since strong combinatorial flattenings always exists,
		there is no loss of generality in using them in the statement.
		\item\label{case:k>1}
		Since $k=1$, the term $2^k$ in the denominator of \eqref{1-loop-formula} simply
		equals $2$.
		Under certain additional assumptions on $\triang$, a generalization of
		the \nword{1}{loop} Conjecture to the case of multiple toroidal ends ($k>1$) is
		discussed in \cite[Section~5.5.2]{rs-phd}.
		This generalization involves the factor $2^k$ and the
		matrices~$\hat{G}^\square$ defined by \eqref{def-hats} with a general $k$.
	\end{enumerate}
\end{remark}

\section{Geometric computation of the torsion}\label{torsion-section}
\subsection{Factorization of torsion with respect to an ideal triangulation}\label{dual-complex}
A positively oriented geometric ideal triangulation~$\triang$ of a finite-volume
hyperbolic \nword3manifold~$M$ defines a finite cell complex~$X$ dual to $\triang$.
As usual, the \nword0cells of $X$ are in a bijective correspondence with the tetrahedra of
$\triang$.
There is a \nword1cell in $X$ for every face of the triangulation and a \nword2cell for every edge.
In this way, the \nword2dimensional CW-complex $X$ represents the homotopy type of $M$.

We shall use the cochain complex $C^\bullet(X;\KKsub{X})$ to calculate the adjoint
torsion~$\ATor(M,\theta)$.
Observe that the subspace~$M_0\subset M$ corresponds to the \nword1skeleton~$X^{(1)}\subset X$,
so that the long exact sequence~\eqref{reduced-long-exact} can be constructed in cellular
cohomology of $\KKsub{X}$ as the long exact sequence of the pair $(X,X^{(1)})$.

We fix the basis $\{\slE, \slH, \slF\}\subset\slC$, where
\(
	\slE = \left[\begin{smallmatrix}0 & 1 \\ 0 & 0\end{smallmatrix}\right],
	\slH = \frac12\left[\begin{smallmatrix}1 & 0 \\ 0 & -1\end{smallmatrix}\right],
	\slF = \left[\begin{smallmatrix}0 & 0 \\ 1 & 0\end{smallmatrix}\right].
\)
We are now going to construct a particularly convenient class of geometric
bases~$\basis{c}^\bullet_{\text{geom}}$ of the cellular cochain complex~$C^\bullet(X;\KKsub{X})$.
Given a \nword2cell $s_i$ of $X$ dual to an edge~$e_i$ of $\triang$,
choose a local coordinate chart $\varphi_i$ which maps the edge $e_i$
to the geodesic $(0,\infty)\subset\HH$.
We then define a geometric basis of $C^2(X;\KKsub{X})$ using \eqref{really-geometric-basis}
with $\varphi=\varphi_i$ whenever $s=s_i$, $1\leq i\leq N$.
Thanks to the factor of $\frac{1}{2}$ in the definition of $\slH$,
equation \eqref{calculation-of-t} shows that the local Killing field $(\varphi_i^{-1})_*(\slH)$
equals $t(e_i,\varepsilon)$, where the orientation $\varepsilon\in\Or(e_i)$
is dual to the orientation of $s_i$.
From now on, we assume that the geometric basis~$\basis{c}^\bullet_{\text{geom}}$ is constructed as
above.
We do not impose any special requirements on the charts $\varphi$ for cells in dimensions zero and one.

Applying Milnor's Multiplicativity Theorem~\cite[Theorem~3.2]{milnor-whitehead} to the short
exact sequence of cochain complexes given by the pair $(X,X^{(1)})$,
we obtain the decomposition
\begin{multline}\label{first-decomposition}
	\ATor(M,\theta)
	= \Tor\bigl( C^\bullet(X; \KK), \basis{c}^\bullet_{\text{geom}}, \basis{h}^\bullet(\theta)\bigr)\\
	= \Tor\bigl( C^\bullet(X^{(1)};\KK), \basis{c}^{(0,1)}_{\text{geom}}, \basis{h}^1_0\bigr)
	\Tor\bigl(C^\bullet(X, X^{(1)};\KK), \basis{c}^{(2)}_{\text{geom}}, \basis{h}^2_{\text{rel}}\bigr)
	\Tor\bigl(\mathcal{H}^\bullet,
		\basis{h}^\bullet(\theta)\cup\basis{h}^1_0\cup\basis{h}^2_{\text{rel}},\varnothing\bigr),
\end{multline}
where $\basis{h}^\bullet(\theta)$ is any cohomology basis balanced with respect to $\theta$,
$\basis{h}^1_0$ is any basis of $H^1(X^{(1)};\KK)=H^1(M;\KKsub{M_0})$ and
$\basis{h}^2_{\text{rel}}$ is any basis of $H^2(X,X^{(1)};\KK)=H^2(M;\KKsub{M,M_0})$.
In the above formula, the symbol~$\mathcal{H}^\bullet$ denotes the cohomology long exact
sequence~\eqref{reduced-long-exact}.
Note that the bases $\basis{h}^1_0$ and $\basis{h}^2_{\text{rel}}$ occurring in
\eqref{first-decomposition} can be chosen arbitrarily, but
we shall indicate a particularly good choice in Section~\ref{good-bases}.

Suppose that the log-parameter $u=u_\theta$ satisfies $0<|u_l|<\pi$ for all $l$,
so that the incomplete hyperbolic structure on $M$ is a small deformation of the unique complete structure.
By Theorem~\ref{commutative-diagram-theorem}, we can apply Milnor's theorem to the last
term  of \eqref{first-decomposition}, obtaining
\begin{equation}\label{second-decomposition}
\Tor\bigl(\mathcal{H}^\bullet,
	\basis{h}^\bullet(\theta)\cup\basis{h}^1_0\cup\basis{h}^2_{\text{rel}},\varnothing\bigr)
= \Tor\bigl(\mathcal{G}^\bullet, \Choibasis,\varnothing\bigr)
  \Tor\bigl(\Coker D\paraZ \to \Coker\alpha, \basis{q^\bullet},\varnothing\bigr),
\end{equation}
where $\mathcal{G}^\bullet$ is the `gluing complex' \eqref{Choi-choice} and the bases~$\Choibasis$,
$\basis{q}^\bullet$ must be chosen so as to satisfy the graded compatibility assumptions of Milnor's
theorem.
We construct such bases in the next section.

\subsection{Construction of graded-compatible bases}\label{good-bases}
We wish to equip the gluing complex~$\mathcal{G}^\bullet$ of \eqref{Choi-choice}
with a basis $\Choibasis$ given by the partial derivatives with respect to the standard coordinates.
More precisely, take $1\in\CC$ as the basis vector of the last term $\CC^k=\CC$
and pick the standard basis vectors
$\partial/\partial u \in T_u U$,
$\partial/\partial z_1, \dotsc, \partial/\partial z_N\in T_z\CC_{\Imag>0}^N$,
$\partial/\partial x_1,\dotsc \partial/\partial x_N\in T_1(\CC^*)^N$.
This defines the basis~$\Choibasis$ which we use in \eqref{second-decomposition}.

Observe that the maps~$D\paraX$ and~$\beta$ of Theorem~\ref{commutative-diagram-theorem}
define cohomology basis vectors
\begin{equation}\label{Choi-balanced}
D\paraX\bigl({\textstyle\frac{\partial}{\partial u}}\bigr)\in H^1(M;\KKsub{M})
\quad\text{and}\quad
\beta(1)\in H^2(M;\KKsub{M}).
\end{equation}
An application of Theorem~\ref{log-parameters-cups}
implies that the vectors~\eqref{Choi-balanced} form a cohomology
basis balanced with respect to the curve $\theta$ in the sense of
Definition~\ref{Ad-torsion}-\eqref{balanced}.
We therefore define $\basis{h}^\bullet(\theta)$ to consist of the elements given in \eqref{Choi-balanced}.

Next, we define the basis $\basis{h}^2_{\text{rel}}\subset H^2(M;\KKsub{M,M_0})$ as the image of
$\basis{c}^{(2)}_{\text{geom}}$ under the canonical
isomorphism~$C^2(X,X^{(1)};\KK)\cong H^2(M;\KKsub{M,M_0})$.
This choice of bases ensures that
\begin{equation}\label{relative-1}
	\Tor\bigl(C^\bullet(X, X^{(1)};\KK),\basis{c}^{(2)}_{\text{geom}},\basis{h}^2_{\text{rel}} \bigr)
	=\pm 1.
\end{equation}
By our assumption on the choice of the local charts $\varphi_i$, we see that
the set~$\bigl\{\alpha\bigl({\textstyle\frac{\partial}{\partial x_i}}\bigr)\bigr\}_{i=1}^N$
is a subset of $\basis{h}^2_{\text{rel}}$.
The remaining part of $\basis{h}^2_{\text{rel}}$ must therefore descend to a basis
of $\Coker\alpha$, which we denote by $\basis{q^2}$.
Using the notations of \eqref{really-geometric-basis},
we can describe the basis~$\basis{q^2}$ explicitly: it is induced by cochains
of the form $s\mapsto(\varphi^{-1})_*(\slE)$, $s\mapsto(\varphi^{-1})_*(\slF)$,
where $s$ runs over the set of oriented \nword{2}{cells} dual to edges of $\triang$.
In this way, the three bases~$\bigl\{\frac{\partial}{\partial x_i}\bigr\}_i$, $\basis{h}^2_{\text{rel}}$
and $\basis{q^2}$ are compatible.

Observe that the cokernel complex~$\Coker D\paraZ\mapswith{[\Delta]}\Coker\alpha$ consists of only
two non-zero terms with the isomorphism $[\Delta]$ induced via
Theorem~\ref{commutative-diagram-theorem} from the connecting homomorphism~$\Delta$.
Hence, it makes sense to set~$\basis{q^1}=[\Delta]^{-1}(\basis{q^2})$.
This choice of $\basis{q^\bullet}$ ensures that
\begin{equation}\label{cokernel-1}
	\Tor\bigl(\Coker D\paraZ \to \Coker\alpha, \basis{q^\bullet},\varnothing\bigr) = \pm 1.
\end{equation}
Finally, we choose a collection~$\widetilde{\basis{q^1}}$ of lifts of the vectors of $\basis{q^1}$ to
$H^1(M;\KKsub{M_0})$ and set
\begin{equation}\label{h1_M0}
\basis{h}^1_0 := \widetilde{\basis{q^1}}\cup
\bigl\{D\paraZ\bigl({\textstyle\frac{\partial}{\partial z_j}}\bigr)\bigr\}_{j=1}^N
	\subset H^1(M;\KKsub{M_0}).
\end{equation}
This choice guarantees that $\basis{h}^1_0$,
$\bigl\{\dby{z_j}\bigr\}_{j=1}^N$
and $\basis{q^1}$ also satisfy the compatibility condition.

\subsection{Reduction of the 1-loop Conjecture}
Using the graded-compatible bases constructed in the preceding section,
we can use the decompositions~\eqref{first-decomposition} and~\eqref{second-decomposition}
to compute the adjoint hyperbolic torsion $\ATor(M,\theta)\in\CC^*/\{\pm1\}$.
Thanks to~\eqref{cokernel-1} and~\eqref{relative-1}, we can write
\begin{equation}\label{final-decomposition}
	\ATor(M,\theta)
	= \pm
 	\Tor\bigl(\mathcal{G}^\bullet, \Choibasis,\varnothing\bigr)
	\Tor\bigl( C^\bullet(X^{(1)};\KK), \basis{c}^{(0,1)}_{\text{geom}}, \basis{h}^1_0\bigr)
	=: \pm \Tor_1 \Tor_2.
\end{equation}
We believe that it is possible to express the two factors $\Tor_1$ and $\Tor_2$ in closed form
in terms of the so-called \emph{enhanced Neumann--Zagier datum}~$(\hat{G}, \hat{G}', \hat{G}'', z, f,f',f'')$.
The lemma given below substantiates this belief in the case of $\Tor_1$.
\begin{lemma}\label{calculated-Choi-torsion}
\(\displaystyle
	\Tor_1
	= \pm{\textstyle\frac{1}{2}}\det\bigl(
	\hat{G}\diag(\zeta) + \hat{G}'\diag(\zeta') + \hat{G}''\diag(\zeta'')\bigr).
\)
\end{lemma}
We postpone the proof of this lemma until Section~\ref{torsion-Choi} below.
\begin{corollary}
When $\theta$ and $\widetilde{\theta}$ are two homotopically non-trivial
simple closed curves in the boundary torus $\partial\overline{M}$,
then
\[
	\frac{\ATor(M,\widetilde{\theta})}
	     {\ATor(M,\theta)}
	= \pm \frac{\det \bigl(
	  \widetilde{\hat{G}}\diag(\zeta)
	+ \widetilde{\hat{G'}}\diag(\zeta')
	+ \widetilde{\hat{G''}}\diag(\zeta'')\bigr)}
	     {\det \bigl(
	  \hat{G}\diag(\zeta)
	+ \hat{G}'\diag(\zeta')
	+ \hat{G}''\diag(\zeta'')\bigr)},
\]
where the matrices $(\hat{G},\hat{G'},\hat{G''})$
and $(\widetilde{\hat{G}},\widetilde{\hat{G'}},\widetilde{\hat{G''}})$
contain in their last rows the coefficients of the completeness equation along $\theta$
and $\widetilde{\theta}$, respectively.
\end{corollary}
\begin{proof}
This follows at once from \eqref{final-decomposition}, Lemma~\ref{calculated-Choi-torsion},
and from the fact that $\Tor_2$ does not depend on the choice of the boundary curve.
\end{proof}

Comparing the expression for $\Tor_1$ given in Lemma~\ref{calculated-Choi-torsion} with the \nword{1}{loop}
formula~\eqref{1-loop-formula}, we obtain the following reduction of Conjecture~\ref{1-loop}.

\begin{conjecture}[Reduced \nword1loop Conjecture]\label{reduced-conjecture}
Whenever $M$ is a connected, orientable hyperbolic \nword{3}{manifold} of finite volume
with $k>0$ toroidal ends equipped with a geometric, positively oriented ideal
triangulation~$(\triang, z)$ and a strong combinatorial flattening~$(f,f',f'')$
on $\triang$, we have
\begin{equation}\label{reduced-formula}
	\Tor_2 = \pm\zeta^{-f} {\zeta'}^{-f'} {\zeta''}^{-f''}.
\end{equation}
\end{conjecture}

Note that the decomposition~\eqref{final-decomposition} holds \emph{a priori}
for incomplete hyperbolic structures
obtained as small deformations of the unique complete structure.
However, it is shown in \cite{tudor-stavros} that
the shape parameters are rational functions on the geometric component~$X_0$ of the
\nword{\PSL}{character} variety~$X(\pi_1(M),\PSL)$.
Hence, Lemma~\ref{calculated-Choi-torsion} implies that $\Tor_1$
defines a rational function (up to sign) on a regular neighborhood of the holonomy
representation of the complete structure.
By a result of Porti~\cite[Proposition~4.14]{porti1997}, the adjoint torsion~$\ATor(M,\theta)$
is also a rational function on~$X_0$.
Hence, the equality \eqref{final-decomposition} guarantees that $\Tor_2$ is
rational as well.
In particular, the decomposition~\eqref{final-decomposition} extends to the holonomy
representation of the complete hyperbolic structure.

\begin{theorem}\label{reduction-theorem}
If $M$, $\triang$ and $(f,f',f'')$ satisfy the assumptions of Conjecture~$\ref{1-loop}$, then
\begin{enumerate}[(i)]\itemsep0pt
	\item\label{conj-reduction}
	Conjecture~$\ref{reduced-conjecture}$ implies the \nword1loop Conjecture~$\ref{1-loop}$ for
	all curves $\theta$;
	\item\label{nonvanishing}
	The conjectural expression~$\eqref{1-loop-formula}$ does not vanish.
\end{enumerate}
\end{theorem}
\begin{proof}
Part \eqref{conj-reduction} follows from the decomposition of torsion discussed above
and from the fact that the Reduced Conjecture~\ref{reduced-conjecture} does not involve the choice
of a peripheral curve.

In order to prove part~\eqref{nonvanishing}, observe that eq.~\eqref{final-decomposition} defines
the term~$\pm\Tor_1$ as the combinatorial torsion of an acyclic complex,
so by \eqref{tordef}, $\Tor_1$ is always a non-zero scalar (defined up to sign).
By Lemma~\ref{calculated-Choi-torsion}, $\Tor_1$ coincides with
the numerator of \eqref{1-loop-formula}.
\end{proof}

\subsection{Torsion of the infinitesimal gluing equations}\label{torsion-Choi}
This section is devoted to the proof of Lemma~\ref{calculated-Choi-torsion}.
Our task is to compute the torsion of the acyclic `tangential gluing complex'
of eq.~\eqref{Choi-choice},
\begin{equation}\label{graded-Choi}
	\begin{tikzcd}[row sep=-1.5ex]
	0 \arrow{r}
	& T_u U \arrow{r}{Dy}
	& T_{y(u)} \CC_{\Imag>0}^N \arrow{r}{Dg}
	& T_1 (\CC^*)^N \arrow{r}{Dp}
	&\CC\ \arrow{r}
	& 0,\\
	\scriptstyle
	& 0 & 1 & 2 & 3 &
	\end{tikzcd}
\end{equation}
with the grading indicated under each non-trivial term, with respect to the bases
\[	
	\basis{c^0}=\bigl\{{\textstyle\frac{\partial}{\partial u}}\bigr\},\quad
	\basis{c^1}=\bigl\{{\textstyle\frac{\partial}{\partial z_j}}\bigr\}_{j=1}^N,\quad
	\basis{c^2}=\bigl\{{\textstyle\frac{\partial}{\partial x_i}}\bigr\}_{i=1}^N,\quad
	\basis{c^3}=\{1\}.
\]
According to \eqref{tordef}, the torsion of \eqref{graded-Choi} can be calculated
as $\pm\frac{\det A_0 \det A_2}{\det A_1 \det A_3}$, where $A_r$ is the
change-of-basis matrix relating the basis~$\basis{b^r}\cup s^r(\basis{b^{r+1}})$
to the basis~$\basis{c^r}$, for every $0\leq r \leq 3$.
We choose the following bases $\basis{b^r}$:
\[
	\basis{b^0} = \varnothing,\quad
	\basis{b^1} = \left\{Dy\bigl(\tfrac{\partial}{\partial u}\bigr)\right\},\quad
	\basis{b^2} = \bigl\{\tfrac{\partial}{\partial x_i}-\tfrac{\partial}{\partial x_N}\bigr\}_{i=1,\dotsc,N-1},\quad
	\basis{b^3} = \{1\}.
\]
We begin the computation with the third term~$\CC$,
where we have $\basis{b^3}=\basis{c^3}=\{1\}$,
so that the change-of-basis matrix equals $A_3=[1]$.

Since $k=1$, all edges of $\triang$ have both their ends incident to the only toroidal end of $M$,
so the Jacobian matrix of the map $p$ of \eqref{def-p} is $[2\ 2 \dotsb 2]$.
Hence, we may choose $\frac{1}{2}\dby{x_N}\in T_1(\CC^*)^N$ as a pre-image of
$1\in\CC$ under $Dp$.
This vector can be completed to a basis by adjoining
the elements of $\basis{b^2}$.
Expressing this basis in terms of the original basis $\basis{c^2}$,
we obtain the change-of-basis matrix
\begin{equation}\label{A2-matrix}
	A_2=\begin{bmatrix}
		1 & 0 & \dots & 0 & 0\\
		0 & 1 &  & 0 & 0\\
		\vdots &&\ddots &&\vdots\\
		0 & 0 &  & 1 & 0\\
		-1 & -1 & \dots & -1 & \frac{1}{2}\\
	\end{bmatrix}.
\end{equation}

Note that $\dby{x_i}-\dby{x_N}\in\Kernel Dp=\Imag Dg$ for every $i$.
Hence, there exist vectors~$w_i \in T_{y(u)} \CC_{\Imag>0}^N$ such that
$Dg(w_i)=\dby{x_i}-\dby{x_N}$ for all $i$.
By exactness, the set~$\{w_1,\dotsc,w_{N-1}\}\subset T_{y(u)}\CC_{\Imag>0}^N$
can be completed to a basis by adjoining the vector~$w_N := Dy\bigl(\dby{u}\bigr)$.
Hence, the change-of-basis matrix $A_1$ has the form
$A_1 = [ w_1| w_2 | \dotsb | w_N]$,
where each column marked $w_i$ contains the coefficients of $w_i$ in
the basis~$\basis{c^1}=\bigl\{\dby{z_j}\bigr\}_{j=1}^N$.
In order to compute the determinant of $A_1$, we need the following lemma.

\begin{lemma}\label{Jacobian-g}
The matrix $G\diag(\zeta)+G'\diag(\zeta')+G''\diag(\zeta'')$
is the Jacobian matrix of the map $g$ of \eqref{def-g} at any point $z\in\glvar$.
Similarly, the matrix $C\diag(\zeta)+C'\diag(\zeta')+C''\diag(\zeta'')$
is the Jacobian of the log-parameter map $u: \CC_{\Imag>0}^N\to\CC$.
\end{lemma}
\begin{proof}
We calculate the $(i,j)$-th entry of the Jacobian of $g$; the proof for $u$ is analogous.
\begin{align*}
	\frac{\partial g_i(z)}{\partial z_j}
	&=
	\Biggl(
		G_{ij}{z_j}^{G_{ij}-1}{z'_j}^{G'_{ij}}{z''_j}^{G''_{ij}}
		+ \frac{G'_{ij}{z_j}^{G_{ij}}{z'_j}^{G'_{ij}-1}{z''_j}^{G''_{ij}}}{(z_j-1)^2}
		+ \frac{G''_{ij}{z_j}^{G_{ij}}{z'_j}^{G'_{ij}}{z''_j}^{G''_{ij}-1}}{z_j^2}
	\Biggr)\\
		&\times\prod_{m\neq j} {z_m}^{G_{im}} {z'_m}^{G'_{im}} {z''_m}^{G''_{im}}\\
	&=
	\Biggl(
			\frac{G_{ij}}{z_j} + \frac{G'_{ij}}{1-z_j} + \frac{G''_{ij}}{z_j(z_j - 1)}
		\Biggr)g_i(z)
	=
		G_{ij}\zeta_j + G'_{ij}\zeta'_j + G''_{ij}\zeta''_j.\qedhere
\end{align*}
\end{proof}
\def\fatG{\mathbb{G}}
Consider the matrices $\hat{G}^\square$ of \eqref{def-hats} and
define $\fatG:=\hat{G}\diag(\zeta) + \hat{G}'\diag(\zeta') + \hat{G}''\diag(\zeta'')$.
By Lemma~\ref{Jacobian-g}, the product $\fatG A_1$ has the block form
\begin{equation}\label{multiplied-by-S}
	\fatG A_1 =
	\left[\;
	\begin{tikzpicture}[baseline=({baza})]
		\node[minimum width=16ex, minimum height=16ex] (Id) at (0,0)
			{$\displaystyle\Id_{(N-1)\times(N-1)}$};
		\draw ({Id.south west}) rectangle ({Id.north east});
		\node[minimum width=16ex,anchor=north,yshift=-2pt] (star) at ({Id.south})
			{$\displaystyle *\vphantom{X}$};
		\draw ({star.south west}) rectangle ({star.north east});
		\node[minimum height=16ex,anchor=west] (jacek) at ({Id.east}) {$\displaystyle\,\fatG w_N$};
		\draw ([xshift=2pt] {star.south east}) rectangle ({jacek.north east});
		\node (baza) at ([yshift=-2.3ex]{Id}) {};
	\end{tikzpicture}\;\right].
\end{equation}

Since the top $N-1$ rows of $\fatG$ agree with those of the Jacobian of $g$ and $Dg\circ Dy=0$,
we find $\fatG w_N = \bigl[0,0,\dotsc,0,1\bigr]^\top$.
This implies that $\fatG A_1$ is a lower-triangular matrix with ones on the main diagonal,
whence $\det A_1 = 1/\det\fatG$.

The remaining change-of-basis matrix at $T_u U$ reduces to $A_0=\left[1\right]$, because
the preimage of $Dy\bigl(\dby{u}\bigr)$ under $Dy$ is tautologically $\dby{u}$.
Hence, the torsion of the complex \eqref{Choi-choice} equals
\begin{equation*}
	\pm\frac{\det A_0 \det A_2}{\det A_1 \det A_3}
	= \pm\frac{1}{2}\det \fatG
	= \pm\frac{1}{2}\det \bigl(
		\hat{G}\diag(\zeta) + \hat{G}'\diag(\zeta') + \hat{G}''\diag(\zeta'')\bigr).
\end{equation*}
This concludes the proof of Lemma~\ref{calculated-Choi-torsion}.

\appendix
\section{The sister manifold of the figure-eight knot complement}
The goal of this appendix is to verify the reduced \nword{1}{loop}
conjecture (Conjecture~\ref{reduced-conjecture})
for the sister manifold of the figure-eight knot complement, denoted by $M$ from now on.
The minimal triangulation~$\triang$ of $M$, with $N=2$ tetrahedra,
can be generated and explored with the help of the computer program Regina~\cite{regina},
which can reconstruct $\triang$ from the isomorphism signature~\texttt{"cPcbbbdxm"}.
In this way, we read off the gluing pattern of $\triang$,
presented in Figure~\ref{fig:triangulation-m003}.
Note that the SnapPea census triangulation~\texttt{m003},
available e.g. from within SnapPy~\cite{snappy},
comes with a slightly different labeling of vertices and peripheral curves,
although it is combinatorially isomorphic to our triangulation~$\triang$.

The complete hyperbolic structure of finite volume on $M$ is recovered when
the shape parameters have values $z_1=z_2 = e^{\pi i/3}$.
However, in this calculation we shall treat the shape parameters as indeterminates,
which will allow us to prove \eqref{reduced-formula} as an equality of rational
functions on the gluing variety~$\glvar$.

Using Figure~\ref{fig:triangulation-m003} or the SnapPea
functionality built into Regina, we can find the gluing matrices of $\triang$;
in the notations of Section~\ref{gluing-section}, they are
\begin{alignat*}{3}
	G&=\begin{bmatrix}
	2 & 1 \\
	0 & 1
	\end{bmatrix},\quad&
	G'&=\begin{bmatrix}
	1 & 0 \\
	1 & 2
	\end{bmatrix},\quad&
	G''&=\begin{bmatrix}
	0 & 2 \\
	2 & 0
	\end{bmatrix},\\
	C&=\begin{bmatrix}
	2 & 0 \\
	-1 & 1
	\end{bmatrix},\quad&
	C'&=\begin{bmatrix}
	0 & -2 \\
	-1 & 0
	\end{bmatrix},\quad&
	C''&=\begin{bmatrix}
	0 & 0 \\
	0 & 1
	\end{bmatrix}.
\end{alignat*}
A strong combinatorial flattening is thus given by $f=(0,1)$, $f'=(1,0)$, $f''=(0,0)$.

Following the approach described in Section~\ref{dual-complex}, we are going to use
the cell complex $X$ dual to $\triang$ as a topological model for $M$.
In our case, $X$ has two \nword{0}{cells} (dual to $\Delta_1$ and $\Delta_2$), four
\nword{1}{cells} (dual to the faces $\mathcal{A}$, $\mathcal{B}$, $\mathcal{C}$,
$\mathcal{D}$) and two \nword{2}{cells}, dual to the edges $e_1$ and $e_2$.
We orient all \nword{1}{cells} to point from $\Delta_1$ to $\Delta_2$
and equip the \nword{2}{cells} with orientations dual to the orientations of the edges $e_1$
and $e_2$ shown as arrows in Figure~\ref{fig:triangulation-m003}, using the right-hand rule.

Recall that the left-hand side~$\Tor_2$ of the conjectural equality \eqref{reduced-formula}
is the Reidemeister torsion of the handlebody~$M_0$ of Definition~\ref{def-M0}.
Since $M_0$ deformation-retracts onto the \nword{1}{skeleton}~$X^{(1)}$,
$\Tor_2$ is the torsion~$\Tor\bigl(C^\bullet(X^{(1)};\KK), \basis{c}^{(0,1)}_{\text{geom}},
\basis{h}^1_0\bigr)$.
Note that the complex $C^\bullet(X^{(1)};\KK)$ has non-vanishing cohomology only in degree~$1$,
so by resolving it on the right, we can also express $\Tor_2$
as the Reidemeister torsion of the acyclic complex
\begin{equation}\label{resolved-complex}
	0\to C^0(X;\KK) \mapswith{\delta^0} C^1(X;\KK)\to H^1(M_0;\KKsub{M_0})\to 0
\end{equation}
with respect to the geometric bases of the cochain spaces, as constructed in
Section~\ref{good-bases}, and the basis
$\basis{h}^1_0 = \bigl\{D\paraZ\bigl({\textstyle\frac{\partial}{\partial z_j}}\bigr)\bigr\}_{j=1}^2
\cup\widetilde{\basis{q^1}}$
for the rightmost term.

\begin{figure}
	\centering
	\begin{tikzpicture}
		\node[anchor=south west,inner sep=0] (image) at (0,0)
		    {\includegraphics[width=0.8\columnwidth]{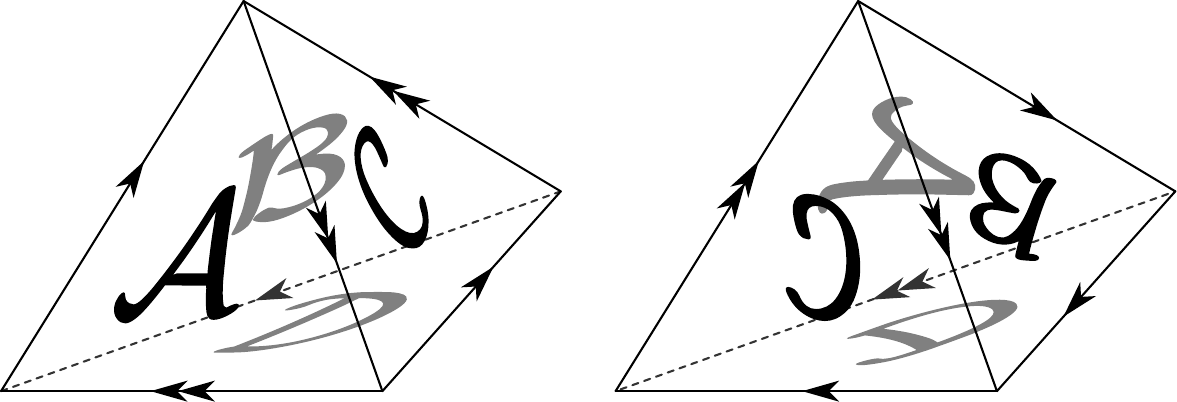}};
		\begin{scope}[x={(image.south east)},y={(image.north west)},
				every node/.style={inner sep=0, outer sep=0}]
			\node[anchor=south] at (0.1, 0.95)   {$\Delta_1$};
			\node[anchor=south] at (0.208, 1.02) {$0$};
			\node[anchor=east] at (0, 0.05)      {$1$};
			\node[anchor=west] at (0.35, 0.05)   {$2$};
			\node[anchor=west] at (0.48, 0.52)   {$3$};
			\node[anchor=south east] at (0.09, 0.5) {$z_1$};
			\node[anchor=north west] at (0.415, 0.27) {$z_1$};
			\node[anchor=east] at (0.265, 0.37) {$z'_1$};
			\node[anchor=north] at (0.2, 0) {$z''_1$};
			\node[anchor=south west] at (0.35, 0.76) {$z''_1$};
			\node[anchor=south] at (0.61, 0.95)  {$\Delta_2$};
			\node[anchor=south] at (0.73, 1.02)  {$0$};
			\node[anchor=east] at (0.522, 0.05)  {$1$};
			\node[anchor=west] at (0.872, 0.05)  {$2$};
			\node[anchor=west] at (1.002, 0.52)  {$3$};
			\node[anchor=south east] at (0.612, 0.5) {$z_2$};
			\node[anchor=north west] at (0.937, 0.27) {$z_2$};
			\node[anchor=east] at (0.787, 0.375) {$z'_2$};
			\node[anchor=north] at (0.742, 0) {$z''_2$};
			\node[anchor=south west] at (0.872, 0.76) {$z''_2$};
		\end{scope}
	\end{tikzpicture}
	\caption{
	A schematic view of the triangulation \texttt{"cPcbbbdxm"}.
	Each of the four faces in the triangulation
	is the result of identifying a pair of faces of the tetrahedra $\Delta_1$, $\Delta_2$ according
	to the gluing pattern shown.
	For example, the face $[0,1,2]$ of $\Delta_1$ is glued to the face $[3,0,1]$ of $\Delta_2$
	(with this order of vertices), thus forming the face $\mathcal{A}$ of $\triang$.
	This face pairing identifies the edges of the tetrahedra in two lots of six, giving rise to the
	two edges $e_1$, $e_2$ of $\triang$, marked here with single and double arrow heads, respectively.
	}\label{fig:triangulation-m003}
\end{figure}

In order to make our computation more explicit, we shall now choose geometric bases for the cochain
groups of $X$ using the method of local geometric coordinate charts described in Section~\ref{geom-bases}.
Let $\varphi_j$ be the local geometric coordinate on the interior of the tetrahedron $\Delta_j$ ($j=1,2$)
uniquely determined by the requirement that its continuation to the face $[0,1,2]$
takes this face to the ideal triangle $(\infty,0,1)\subset\HH$, with vertices corresponding
to one another in the order given.
In this way, the charts $\varphi_1$ and $\varphi_2$ establish the identification $C^0(X;\KK)\cong(\slC)^2$.
Moreover, $\varphi_1$ can be continued from within $\Delta_1$ to interior points of all four faces of $\Delta_1$,
which allows us to identify $C^1(X;\KK)$ with $(\slC)^4$.
For completeness, this construction can be also carried out for the \nword{2}{cells}
dual to the edges: the \nword{2}{cell} dual to $e_j$ carries a local coordinate chart obtained
as the continuation of $\varphi_j$ to the edge $[1,0]$ of $\Delta_j$ ($j=1,2$).
In particular, these charts take both edges to the geodesic $(0,\infty)\subset\HH$, as postulated
in Section~\ref{dual-complex}.
We may now identify $\slC$ with $\CC^3$ by using the standard basis
\begin{equation}\label{ehf}
	\slE = \begin{bmatrix}0 & 1 \\ 0 & 0\end{bmatrix},\quad
	\slH = \frac{1}{2}\begin{bmatrix}1 & 0 \\ 0 & -1\end{bmatrix},\quad
	\slF = \begin{bmatrix}0 & 0 \\ 1 & 0\end{bmatrix}.
\end{equation}

By applying the above basis choices to \eqref{resolved-complex},
we see that $\Tor_2 = \pm(\det A)^{-1}$ for a matrix $A$ with the block
form~$A = \left[\delta^0 \mid L \right]$,
where $\delta^0$ is the matrix of the \nword{0}th coboundary map, while the
columns of $L$ are expressions,
in the basis of $C^1(X;\KK)$, of cochains representing
the cohomology basis vectors making up $\basis{h}^1_0$.
Our task is now to compute $\delta^0$ and $L$.

Using the method detailed in \cite[\S2.3.2]{rs-phd},
we may write down the monodromy of the hyperbolic structure on $M$ in terms of three fundamental
M\"obius transformations: the \emph{complement}~$z\mapsto1-z$, the \emph{inversion}~$z\mapsto z^{-1}$
and the complex \emph{homothety}~$h_s(z)=sz$.
In the basis \eqref{ehf}, the adjoint images of the above M\"obius transformations are
given by
\begin{align}\label{CRH}
C := \Ad(z\mapsto1-z) &= \begin{bmatrix}
		-1	& -1	&  1	\\
		 0	&  1	& -2	\\
		 0	&  0	& -1	
	\end{bmatrix},
	\quad
R := \Ad(z\mapsto z^{-1}) = \begin{bmatrix}
		0	&  0	& 1	\\
		0	& -1	& 0	\\
		1	&  0	& 0	
	\end{bmatrix},
	\\\notag
H_s := \Ad(z\mapsto sz) &= \begin{bmatrix}
		s	& 0	& 0	\\
		0	& 1	& 0	\\
		0	& 0	& s^{-1}	
	\end{bmatrix}, \quad s\in\CC\setminus\{0\},
\end{align}
cf. \cite[eq.~(2.3.7)]{rs-phd}.
For a face $\mathcal{F}$ of $\triang$, we denote by $\mu_{\mathcal{F}}\in\PSL$
the monodromy of the hyperbolic structure along the oriented \nword{1}{cell} dual to $\mathcal{F}$
with respect to the chosen coordinates at both ends.
With this notation, the matrix of $\delta^0$ has the block form
\begin{equation}\label{delta0}
	\delta^0 = \left[\begin{array}{c|c}
	           -\Id & \Ad\left(\mu^{-1}_{\mathcal{A}}\right) \\\hline
	           -\Id & \Ad\left(\mu^{-1}_{\mathcal{B}}\right) \\\hline
	           -\Id & \Ad\left(\mu^{-1}_{\mathcal{C}}\right) \\\hline
	           -\Id & \Ad\left(\mu^{-1}_{\mathcal{D}}\right)
	           \end{array}\right],
\end{equation}
where $\Id = \Id_{3\times3}$ is the identity matrix.
Using Figure~\ref{fig:triangulation-m003} and tracking down the
M\"obius transformations relating the chosen geometric charts, we find
\begin{alignat}{2}\label{monodromies}
	\Ad\left(\mu^{-1}_{\mathcal{A}}\right) &= RCH^{-1}_{z_2}, \qquad &
	\Ad\left(\mu^{-1}_{\mathcal{B}}\right) &= H_{z_1}CRH_{z'_2}C,\\\notag
	\Ad\left(\mu^{-1}_{\mathcal{C}}\right) &= CH^{-1}_{z'_1}C, \qquad &
	\Ad\left(\mu^{-1}_{\mathcal{D}}\right) &= RCH_{z''_1}CH^{-1}_{z''_2}CR.
\end{alignat}

We now turn to the task of determining the $12\times6$ matrix $L$ forming the right half of $A$.
Note that although the choice of $L$ is not unique, the determinant of $A$ will not be affected
by this indeterminacy.
Denote by $l_1,\dotsc,l_6$ the columns of $L$.
For $j\in\{1,2\}$, we shall take $l_j$ to be
a cocycle in $C^1(X;\KK)\cong(\slC)^4$
representing the element $D\paraZ\bigl({\textstyle\frac{\partial}{\partial z_j}}\bigr)$,
and the remaining columns $l_3,\dotsc,l_6$ will likewise correspond to
the elements of $\widetilde{\basis{q^1}}$.

We begin by computing the column vectors $l_1,l_2$ of $L$.
We shall use Weil's method \cite{weil-remarks}, differentiating the monodromies
along the four \nword{1}{cells} of $X$ given in \eqref{monodromies}
with respect to the shape parameters $z_1$, $z_2$.
In this way,
\begin{align*}
	l_1 = \frac{d}{dt}\Biggr|_{t=z_1}
	\Bigl(
		 \mu^{-1}_{\mathcal{A}}(t,z_2)\mu_{\mathcal{A}}(z_1,z_2),\dotsc,
		 \mu^{-1}_{\mathcal{D}}(t,z_2)\mu_{\mathcal{D}}(z_1,z_2)
	\Bigr)^\top
\end{align*}
and analogously for $l_2$. Using the formula
\(
	\frac{d}{ds}\bigr|_{s=z}\left[h_s h^{-1}_z\right] = \frac{d\log z}{dz}\slH
\)
and keeping \eqref{def-zetas} in mind, we obtain
\begin{equation}\label{l1,l2}
\begin{array}{rrrrr}
l_1  = \bigl( & 0, & \zeta_1\slH, & -\zeta'_1 C(\slH), & \zeta''_1 RC(\slH)\bigr)^\top, \\
l_2  = \bigl( & -\zeta_2RC(\slH), & \zeta'_2H_{z_1}CR(\slH), & 0, & -\zeta''_2RCH_{z''_1}C(\slH)\bigr)^\top.
\end{array}
\end{equation}

We are now left with the task of determining the remaining column vectors~$l_3,\dotsc,l_6$.
Recall from Section~\ref{good-bases} that the basis $\basis{q^1}$ of $\Coker D\paraZ$
was constructed as the preimage of the basis $\basis{q^2}$ of $\Coker\alpha$
under the connecting homomorphism.
Using the isomorphism~\eqref{disc-isomorphism},
$\basis{q^2}$ can be easily described in terms of its representing \nword{2}{cochains}
in $C^2(X;\KK)\cong(\slC)^2$: in our situation, it consists of the four elements
$\{(\slE,0),(\slF,0),(0,\slE),(0,\slF)\}$, taken modulo the subspace of $(\slC)^2$
spanned by $(\slH,0)$ and $(0,\slH)$.
Hence, we could define $l_3,\dotsc,l_6$ as pre-images under the
coboundary operator $\delta^1: C^1(X;\KK)\to C^2(X;\KK)$
of these four basis vectors.
This computation can be further simplified if we recall from Section~\ref{torsion-review}
that the Reidemeister torsion depends only on the exterior product of basis elements.
In other words, it suffices for $l_3,\dotsc, l_6$ to satisfy the equation
\begin{equation}\label{wedgeq}
	\pm\delta^1\bigl(l_3\wedge l_4\wedge l_5\wedge l_6\bigr)
	= [(\slE,0)]\wedge[(\slF,0)]\wedge[(0,\slE)]\wedge[(0,\slF)]
	\in \bigwedge\nolimits^{\mkern-4mu 4} (\slC/\linspan\{\slH\})^2.
\end{equation}
We shall now find an expression for the matrix of the coboundary operator $\delta^1$.
For each face~$\mathcal{F}$ of the triangulation~$\triang$, $\delta^1$ will contain three contributions
corresponding to the three edges of $\mathcal{F}$, which need to be expressed in the
basis of $C^2(X;\KK)\cong(\slC)^2$ coming from local charts on $e_1$ and $e_2$.
This computation is greatly aided by Figure~\ref{fig:edges-m003}, which illustrates the incidences of the faces
of the triangulation~$\triang$ to its edges.
Using the method of \cite[\S2.3.2]{rs-phd}, we find the following block form of $\delta^1$:
\begin{equation}
	\delta_1 = \left[
	\begin{array}{c|c|c|c}
	-\Id
	& \mathstos{H^{-1}_{z''_2}CRH^{-1}_{z_1}}{{} + \Id}
	& H_{z_1^2z_2}RCRH_{z'_1}C
	& \mathstos{-H^{-1}_{z'_1 z''_2} C H^{-1}_{z''_1}CR}{-H_{z_1z_2}RCRH^{-1}_{z''_1}CR}
	\\\hline
	\mathstos{-H_{z_2}CR}{-H_{z_2 z''_1 z'_2}RC}
	& -CH_{z'_1}C
	& \mathstos{H_{z'_1 z'_2 z''_1 z_2}R H_{z'_1}C}{{}+CH_{z'_1}C}
	& H_{z_2}CR
	\end{array}\right].
\end{equation}
\begin{figure}
	\centering
	\includegraphics[width=0.8\columnwidth]{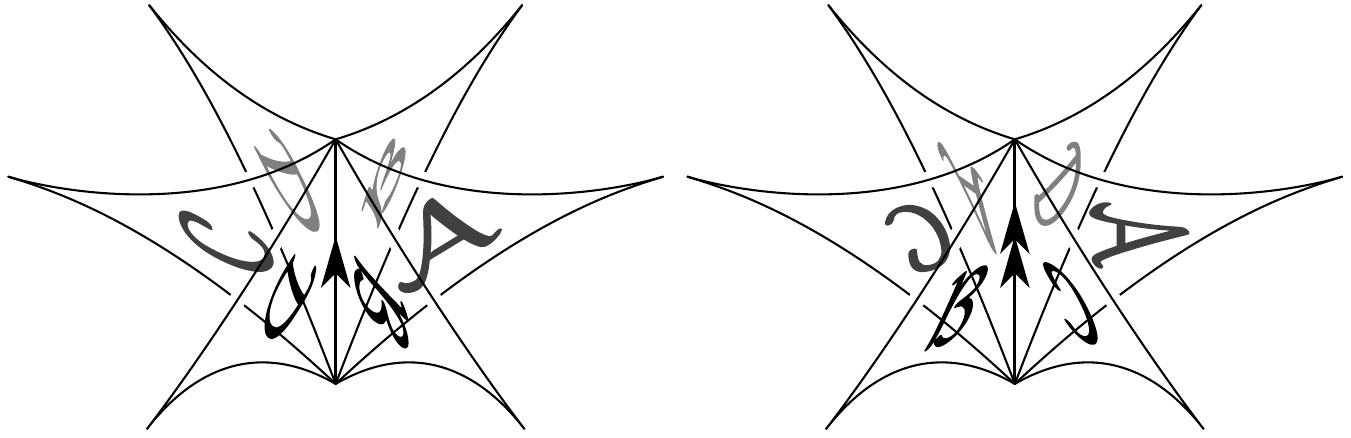}
	\caption{
	Local structure of neighborhoods of the edges $e_1$ and $e_2$ of $\triang$.
	Since every face of the triangulation has three edges, it will account for three
	face-to-edge incidences and thus appear three times in the figure.
	The letters labeling the faces are positioned in the same way as
	in Figure~\ref{fig:triangulation-m003}, and therefore may appear rotated and/or reflected here.
	}\label{fig:edges-m003}
\end{figure}

To make equation~\eqref{wedgeq} more explicit, we may now set
$Q = [l_3\mid l_4\mid l_5\mid l_6]$ and impose the matrix equation
\begin{equation}\label{lifting-equation}
	\delta^1 Q =
	\begin{bmatrix}
	1 & * & * & * \\
	* & * & * & * \\
	0 & 1 & * & * \\
	0 & 0 & 1 & * \\
	* & * & * & * \\
	0 & 0 & 0 & 1
	\end{bmatrix},
\end{equation}
in which the entries marked with ``$*$'' are unimportant.
While solving the above equation is possible by hand, it is perhaps more
convenient to use a computer algebra system.
For the convenience of the reader, the author
provides\footnote{Available at \texttt{http://rs-math.net/attachments/appendix-m003.sage}}
a \texttt{SageMath}~\cite{sagemath} script which verifies the computations presented here.

For any $m,n\in\NN$, denote by $\mathbf{0}_{m\times n}$ the $m\times n$ matrix of zeroes.
A particular solution of eq.~\eqref{lifting-equation} is
$Q=\begin{bmatrix}\mathbf{0}_{3\times4}\\Q'\end{bmatrix}$, where $Q'$ is the $9\times 4$ matrix given by
\begin{equation}\label{Qprime}
Q' = 	\left[\;
	\begin{tikzpicture}[baseline=({Q43.north}), every node/.style={minimum width=4em, minimum height=3ex}]
		\node (Q11) at (0,0) {$\frac{1}{1+z_2(z_1-1)}$};
		\node[right=3em of Q11] (Q12) {$0$};
		\node[right=3em of Q12] (Q13) {$z_1-1$};
		\node[right=3em of Q13] (Q14) {$0$};
		
		\node[below=0.1ex of Q11] (Q21) {$\frac{z'_1 z''_1}{1+z_2(z_1-1)}$};
		\node[right=3em of Q21] (Q22) {$0$};
		\node[right=3em of Q22] (Q23) {$0$};
		\node[right=3em of Q23] (Q24) {$0$};

		\node[below=0.1ex of Q21] (Q31) {};
		\node[right=3em of Q31] (Q32) {};
		\node at ([xshift=4em, yshift=-2ex]{Q31}) {$\mathbf{0}_{5\times2}$};
		\draw ({Q31.north west}) rectangle ([yshift=-4ex, xshift=1em]{Q32.south east});
		\node[right=3.3em of Q32] (Q33) {$0$};
		\node[right=3em of Q33] (Q34) {$(z_1-1)^{-1}$};

		\node[below=0.6ex of Q33] (Q43) {};
		\node[below=0.6ex of Q34] (Q44) {};
		\node at ([xshift=3.7em, yshift=-4.5ex]{Q43}) {$\mathbf{0}_{6\times2}$};
		\draw([xshift=-1em]{Q43.north west}) rectangle ([yshift=-8.5ex]{Q44.south east});

		\node[below=4ex of Q31] (Q51) {$\frac{z_2-1}{1+z_2(z_1-1)}$};
		\node[right=3em of Q51] (Q52) {$z_1z_2$};
		
		\node[below=0.1ex of Q51] (Q61) {$\frac{z_2-1}{1+z_2(z_1-1)}$};
		\node[right=3em of Q61] (Q62) {$z_1z_2$};
	\end{tikzpicture}\;\right].
\end{equation}
In what follows, we may therefore take $L = [l_1 \mid l_2 \mid Q]$, which completes the construction of
the matrix $A=[\delta^0\mid L]$ satisfying $\Tor_2=\pm(\det A)^{-1}$.

Before we compute the determinant of $A$, we can make certain simplifications.
For instance, eq.~\eqref{delta0} shows that the three leftmost columns of $\delta^0$
consist of four copies of the negative identity matrix stacked on top of one another.
After row operations on $A$ which zero out all but the top copy, we obtain
a matrix with the block form $\left[\begin{smallmatrix}-\Id & * \\0 & A'\end{smallmatrix}\right]$, where $A'$
is a $9\times9$ matrix such that $\det A = -\det A'$.
Next, observe that the last four colums of $A'$ are given exactly by the matrix $Q'$ of eq.~\eqref{Qprime}.
As the two rightmost columns of $Q'$ contain only one nonzero entry each, we may consider
the $7\times7$ matrix~$A''$ resulting from striking out the last two columns as well as rows 1 and 3 from
the matrix $A'$.
Since the nonzero entries in the deleted columns of $Q'$ are mutually reciprocal,
we have $\Tor_2 = \pm (\det A'')^{-1}$.
Using \eqref{CRH}--\eqref{monodromies} and applying the aforementioned transformations,
we obtain the explicit expression
\begin{equation}
A''=
\begin{bmatrix}
2z'_2 & 2z'_2 & \frac{2(2-z_2)}{z''_2} & \zeta_1 & \zeta''_2 & \frac{z''_1-1}{1+z_2(z_1-1)} & 0\\
1-z_1 & -z_1 & z_1+\frac{1}{z''_1}+z_2 & \zeta'_1 & 0 & 0 & 0 \\
0 & 2 & -2\bigl(\frac{1}{z''_1}+z_2\bigr) & -\zeta'_1 & -\zeta_2 & 0 & 0 \\
z_2^{-1} & 1 & z'_1-z_2 & 0 & -\zeta_2 & 0 & 0 \\
-\frac{z''_2}{z''_1} & 0 & z_2 & 0 & 0 & 0 & 0 \\
2(1-z'_1z''_2) & 2 & -2z_2 & -\zeta''_1 &
\zeta''_2-\zeta_2
& \frac{z_2-1}{1+z_2(z_1-1)} & z_1z_2 \\
b(z_1,z_2) & 2-z''_1z'_2 & -z_2-\frac{z''_1}{z''_2} & -\zeta''_1
& \zeta_1\zeta''_2-\zeta_2 & \frac{z_2-1}{1+z_2(z_1-1)} & z_1z_2
\end{bmatrix},
\end{equation}
where the bottom left entry is
$b(z_1,z_2)=\frac{1}{z''_1}+\zeta''_2(\zeta''_1-1)-\zeta'_2(\zeta'_1+2)$.
Direct computation shows that
\(
	\Tor_2 = \pm (\det A'')^{-1} = \pm z_2(1-z_1)
\).
At the same time, the right-hand side of the conjectural formula~\eqref{reduced-formula}
is
\[
	\pm\prod_{j=1}^2 {\zeta_j}^{-f_j} {\zeta_j'}^{-f'_j} {\zeta''_j}^{-f''_j}
	= \pm (\zeta'_1)^{-1} (\zeta_2)^{-1} = \pm z_2(1-z_1) = \Tor_2.
\]
This confirms Conjecture~\ref{reduced-conjecture}, implying that
the generalized \nword{1}{loop} conjecture holds
on the entire gluing variety of $\triang$ for any choice of the peripheral curve $\theta$.

\end{document}